\documentclass[12pt]{article}

\usepackage{geometry}
\usepackage{amsthm, amsmath, amsfonts, amssymb, mathtools, mathrsfs}
\usepackage{array}
\usepackage{url, hyperref}
\usepackage{xcolor}
\usepackage{stmaryrd}   
\usepackage{bbm}        
\usepackage{tikz}
\usetikzlibrary{decorations.markings,arrows.meta,calc}
\usepackage{shuffle}    

\newtheorem{theorem}{Theorem}[section]
\newtheorem*{theorem*}{Theorem}
\newtheorem{lemma}[theorem]{Lemma}
\newtheorem{proposition}[theorem]{Proposition}
\newtheorem{corollary}[theorem]{Corollary}

\theoremstyle{remark}
\newtheorem{definition}[theorem]{\bf Definition}
\newtheorem{remark}[theorem]{\bf Remark}
\newtheorem{example}[theorem]{\bf Example}

\newtheorem{conjecture}[theorem]{\bf Conjecture}
\newtheorem*{acknowledgements}{\bf Acknowledgements}

\makeatletter
\newcommand{\subjclass}[2][2020]{%
  \let\@oldtitle\@title%
  \gdef\@title{\@oldtitle\footnotetext{#1 \emph{Mathematics subject classification}: #2}}%
}
\newcommand{\keywords}[1]{%
  \let\@@oldtitle\@title%
  \gdef\@title{\@@oldtitle\footnotetext{\emph{Key words and phrases}: #1}}%
}
\makeatother
\renewcommand{\ge}{\geqslant}
\renewcommand{\le}{\leqslant}

\newcommand{\Cpp}{C\nolinebreak\hspace{-.05em}\raisebox{.4ex}{\tiny\bf +}\nolinebreak\hspace{-.10em}\raisebox{.4ex}{\tiny\bf +}}

\newcommand{\Span}{\operatorname{Span}}

\newcommand{\cI}{\mathcal{I}}
\newcommand{\cZ}{\mathcal{Z}}
\newcommand{\cM}{\mathcal{M}}
\newcommand{\cB}{\mathcal{B}}
\newcommand{\cU}{\mathcal{U}}
\newcommand{\cN}{\mathcal{N}}
\newcommand{\cJ}{\mathcal{J}}

\newcommand{\wt}{\operatorname{wt}}
\newcommand{\dep}{\operatorname{dep}}

\newcommand{\LiA}{\operatorname{Li}_A}
\newcommand{\LiAS}{\operatorname{Li}_A^{\star}}

\newcommand{\bbF}{\mathbb{F}}
\newcommand{\bbZ}{\mathbb{Z}}
\newcommand{\bbQ}{\mathbb{Q}}
\newcommand{\bbR}{\mathbb{R}}

\newcommand{\Rw}{\mathscr{R}_w^{(\mathbb{F}_q)}}
\newcommand{\RwFD}{\mathscr{R}_w^{(\mathbb{F}_q[D_1])}}
\newcommand{\RwK}{\mathscr{R}_w^{(K)}}
\newcommand{\Qw}{\mathscr{Q}_w^{(\mathbb{F}_q)}}
\newcommand{\Uw}{\mathscr{U}_w^{(\mathbb{F}_q)}}
\newcommand{\Sw}{\mathscr{S}_w^{(\mathbb{F}_q)}}

\newcommand{\fs}{\mathfrak{s}}
\newcommand{\fh}{\mathfrak{h}}
\newcommand{\fm}{\mathfrak{m}}
\newcommand{\ft}{\mathfrak{t}}
\newcommand{\fu}{\mathfrak{u}}
\newcommand{\fv}{\mathfrak{v}}
\newcommand{\fw}{\mathfrak{w}}
\newcommand{\fg}{\mathfrak{g}}
\newcommand{\fa}{\mathfrak{a}}
\newcommand{\fb}{\mathfrak{b}}

\newcommand{\bS}{\boldsymbol{S}}
\newcommand{\bR}{\boldsymbol{R}}
\newcommand{\brho}{\boldsymbol{\rho}}
\newcommand{\bsigma}{\boldsymbol{\sigma}}
\newcommand{\bQ}{\boldsymbol{Q}}
\newcommand{\bzero}{\boldsymbol{0}}
\newcommand{\bU}{\boldsymbol{U}}
\newcommand{\bF}{\boldsymbol{F}}
\newcommand{\bc}{\boldsymbol{c}}
\newcommand{\bA}{\boldsymbol{A}}
\newcommand{\bB}{\boldsymbol{B}}

\newcommand{\bH}{\boldsymbol{H}}
\newcommand{\bM}{\boldsymbol{M}}
\newcommand{\bT}{\boldsymbol{T}}

\newcommand{\sfE}{\mathsf{E}}
\newcommand{\sfC}{\mathsf{C}}
\newcommand{\sfL}{\mathsf{L}}
\newcommand{\sfZ}{\mathsf{Z}}
\newcommand{\sfD}{\mathsf{D}}
\newcommand{\sfLd}{\widehat{\mathsf{L}}}

\newcommand{\Ind}{\operatorname{Ind}}
\newcommand{\gInd}{\operatorname{gInd}}
\newcommand{\Indbar}{\overline{\operatorname{Ind}}}

\newcommand{\gIKNw}{\mathcal{R}_w^{(\textup{gIKN})}}
\newcommand{\IKNw}{\mathcal{R}_w^{(\textup{IKN})}}
\newcommand{\HLw}{\mathcal{Q}_w}

\newcommand{\BHw}{\mathcal{B}_w^{(\textup{H})}}
\newcommand{\BTTw}{\mathcal{B}_w^{(\textup{TT})}}

\numberwithin{equation}{section}

\allowdisplaybreaks

\title{$\mathbb{F}_q$-linear relations among Thakur's multiple zeta values in positive characteristic}
\author{Jinyuan Hu, Hanqing Huang, Li Lai, and Kunyue Li}
\date{}
\subjclass[2020]{11M32, 11R58, 11J93.}
\keywords{Thakur's multiple zeta values, the Carlitz multiple polylogarithm values, $\bbF_q$-linear relations, quadruple-carry relations.}

\begin{document}

\maketitle

\begin{abstract}
Let $\mathcal{Z}_w^{(\mathbb{F}_q)}$ be the $\mathbb{F}_q$-linear subspace of $\mathbb{F}_q(\!(\theta^{-1})\!)$ spanned by Thakur's multiple zeta values $\zeta_A(\mathfrak{s})$ of weight $w$. 
We prove that $\sum_{w=1}^{\infty} \left(\dim_{\mathbb{F}_q} \mathcal{Z}_w^{(\mathbb{F}_q)}\right) x^w = \frac{x(1-x^q)(1-2x+x^q)}{(1-2x+x^{q+1})^2}$. 
Moreover, we construct an explicit $\mathbb{F}_q$-basis of $\mathcal{Z}_w^{(\mathbb{F}_q)}$, and prove that any $\mathbb{F}_q$-linear relation among Carlitz multiple polylogarithm values $\operatorname{Li}_A(\mathfrak{s})$ is an $\mathbb{F}_q$-linear combination of quadruple-carry relations.
This result can be regarded as an $\mathbb{F}_q$-analogue of the corresponding $\mathbb{F}_q(\theta)$-theorem proved by Chang--Chen--Mishiba and independently by Im--Kim--Le--Ngo Dac--Pham.

Our discovery of the quadruple-carry relations is inspired by the recent work of Im--Kim--Ngo Dac. 
These relations may be viewed as $\mathbb{F}_q$-analogues of the double-shuffle relations among classical multiple zeta values $\zeta(\mathfrak{s})$.
\end{abstract}

\section{Introduction}

For any non-empty set $S$, we denote by $S^{\bullet}$ the disjoint union $\{\emptyset\} \sqcup_{k=1}^{\infty} S^k$. 
Let $\bbZ_+$ be the set of positive integers, and denote by $\cI$ the set $\bbZ_+^{\bullet}$.
The elements in $\cI$ are called \emph{tuples}.
For any tuple $\fs =(s_1,\ldots,s_k) \in \cI$, the \emph{depth} $\dep(\fs)$ and \emph{weight} $\wt(\fs)$ are defined by $\dep(\fs) = k$ and $\wt(\fs)=\sum_{i=1}^{k} s_i$, with the convention $\dep(\emptyset)=0$ and $\wt(\emptyset)=0$.
Define $\cI_w = \left\{ \fs \in \cI \mid \wt(\fs) = w \right\}$ for any weight $w \ge 0$, and define $\cI_{>0} = \cI \setminus \{\emptyset\}$.
As usual, $[n]$ denotes the set $\{1,2,\ldots,n\}$ for any $n \in \bbZ_+$, and $[0]=\emptyset$.

\subsection{Classical multiple zeta values}

For any tuple $\fs = (s_1,\ldots,s_k) \in \cI_{>0}$ such that $s_1>1$, the \emph{multiple zeta value} (abbreviated as MZV) $\zeta(\fs)$ is defined by the convergent nested sum
\[
\zeta(\fs) \coloneq \sum_{\substack{n_1 > n_2 > \cdots > n_k \\ n_1,n_2,\ldots,n_k \in \bbZ_+}} \frac{1}{n_1^{s_1}n_2^{s_2}\cdots n_k^{s_k}} \in \bbR.
\]
By convention, $\zeta(\emptyset) \coloneq 1$. 
These values generalize the special values $\zeta(s)$ for integers $s > 1$ of the Riemann zeta function.

The story of MZVs dates back to Euler; he studied MZVs of depth two (double zeta values) and obtained many results, including the famous formula $\zeta(2,1)=\zeta(3)$.
In the early 1990s, the work of Hoffman \cite{Hoffman1992} and Zagier \cite{Zagier1994} sparked widespread interest in the study of MZVs.

Let $\cZ_w^{(\bbQ)}$ be the $\bbQ$-linear subspace of $\bbR$ spanned by MZVs of weight $w$. 
Among the numerous fascinating aspects of MZV research, the following conjecture stands out as one of the most prominent and attractive open problems:

\begin{conjecture}[Zagier--Hoffman]\label{conj_ZH}
For any weight $w \ge 1$, we have $\dim_{\bbQ} \cZ_w^{(\bbQ)} = d_w^{(\bbQ)}$, where the sequence $\{d_w^{(\bbQ)}\}_{w \ge 1}$ is defined by the generating function
\[
\sum_{w=1}^{\infty} d_w^{(\bbQ)} x^w = \frac{x^2(1+x)}{1-x^2-x^3}. 
\]
Moreover, $\left\{ \zeta(\fs) ~\big|~ \fs \in \BHw \right\}$ is a $\bbQ$-basis of $\cZ_w^{(\bbQ)}$, where 
\[
\BHw = \{2,3\}^{\bullet} \cap \cI_w = \left\{ (s_1,\ldots,s_k) \in \cI_w ~\big|~ k \in \bbZ_{+}, s_i \in \{2,3\} \text{ for each } i \in [k] \right\}.
\]
\end{conjecture}

Terasoma \cite{Terasoma2002}, and independently, Deligne and Goncharov \cite{DG2005}, proved that $\dim_{\bbQ} \cZ_w^{(\bbQ)} \le d_w^{(\bbQ)}$.
In 2012, Brown \cite{Brown2012} proved that $\left\{ \zeta(\fs) ~\big|~ \fs \in \BHw \right\}$ is a generating set of $\cZ_w^{(\bbQ)}$.
Thus, the algebraic part of Conjecture \ref{conj_ZH} has been resolved.
However, the transcendental part of Conjecture \ref{conj_ZH} seems out of reach.
We do not even know the $\bbQ$-linear independence of $\zeta(2,3)$ and $\zeta(3,2)$.

The most natural $\bbQ$-linear relations among MZVs are double-shuffle relations.
They arise by comparing the shuffle product with the stuffle product.
Since divergent sums such as ``$\zeta(1)$'' occur, one needs the concept of regularization to include all linear relations among MZVs.
This leads to the extended double-shuffle relations; see \cite{IKZ2006}.
Let $\mathscr{R}_w^{(\bbQ)}$ be the space of $\bbQ$-linear relations among MZVs of weight $w$.

\begin{conjecture}\label{conj_DS}
The space $\mathscr{R}_w^{(\bbQ)}$ is generated by the extended double-shuffle relations.
\end{conjecture}

For the recent progress related to Conjecture \ref{conj_DS}, see \cite{HMSW2025+}.
A systematic introduction to MZVs can be found in the monographs \cite{BF, Zhao2016}.

\subsection{Thakur's multiple zeta values in positive characteristic}

Let $A=\bbF_q[\theta]$ be the polynomial ring in one variable $\theta$ over a finite field $\mathbb{F}_q$ of characteristic $p$.
Denote by $A_+$ the set of monic polynomials in $A$.
Let $K=\bbF_q(\theta)$ be the fraction field of $A$ and $K_{\infty}=\bbF_q(\!(1/\theta)\!)$ the completion of $K$ with respect to the absolute value $|P/Q|_{\infty}=q^{\deg P-\deg Q}$ (where $P,Q \in A$ and $Q \neq 0$).

In 2004, Thakur \cite{Thakur2004} defined the multiple zeta values $\zeta_A(\mathfrak{s})$ in positive characteristic. 
For any $\fs=(s_1,\ldots,s_k) \in \cI_{>0}$, 
\[
\zeta_A(\fs) \coloneq \sum_{\substack{a_1, a_2, \ldots, a_k \in A_+ \\  \deg a_1 > \deg a_2 > \cdots > \deg a_k}} \frac{1}{a_1^{s_1}a_2^{s_2}\cdots a_k^{s_k}} \in K_{\infty}.
\]
By convention, $\zeta_A(\emptyset) \coloneq 1$.

The function field counterpart of Conjecture \ref{conj_ZH} was formulated by Todd \cite{Todd2018} and Thakur \cite{Thakur2017}.
It has been largely solved by Ngo Dac \cite{Ngo2021} in 2021, and completely solved by Im--Kim--Le--Ngo Dac--Pham \cite{IKLNP2024}, independently by Chang--Chen--Mishiba \cite{CCM2023}.
Define $\cZ_{w}^{(K)}$ to be the $K$-linear subspace of $K_\infty$ spanned by Thakur's MZVs $\zeta_A(\mathfrak{s})$ of weight $w$.

\begin{theorem}[Chang--Chen--Mishiba \cite{CCM2023}, independently Im--Kim--Le--Ngo Dac--Pham \cite{IKLNP2024}]\label{thm_CCMIKLNP}
For any weight $w \ge 1$, we have $\dim_{K} \cZ_w^{(K)} = d_w^{(K)}$, where the sequence $\{d_w^{(K)}\}_{w \ge 1}$ is defined by the generating function 
\[
\sum_{w=1}^{\infty} d_w^{(K)} x^w = \frac{x(1-x^{q-1})}{1-2x+x^{q+1}}.
\]
Moreover, $\left\{ \zeta_A(\fs) ~\big|~ \fs \in \BTTw  \right\}$ is a $K$-basis of $\cZ_w^{(K)}$, where 
\[
\BTTw = \left\{ (\fh,z) \in \cI_w ~\Big|~ \fh \in [q]^{\bullet}, z \in [q-1] \right\}.
\]
\end{theorem}

As with many other results in function field arithmetic, the celebrated criterion of Anderson--Brownawell--Papanikolas \cite{ABP2004} is useful in the proof of Theorem \ref{thm_CCMIKLNP}.

Although there are plenty of $K$-linear relations among Thakur's MZVs, only recently have some outstanding $K$-linear relations been found by Im--Kim--Ngo Dac \cite[Proposition 4.3]{IKN2026+}.
For reasons we will explain later in Sections \ref{subsec_IKN} and \ref{subsec_Ind}, we call them special carry relations, or Im--Kim--Ngo Dac relations.
Let $\RwK$ be the space of $K$-linear relations among the Carlitz multiple polylogarithm values of weight $w$ (see definition below).

\begin{theorem}[Im--Kim--Ngo Dac \cite{IKN2026+}]\label{thm_RwK}
The special carry relations of weight $w$ constitute a $K$-basis of $\RwK$.
\end{theorem}

Let us turn to another aspect of Thakur's MZVs.
Instead of the field $K$, one may also consider the linear relations over $\bbF_q$.
Already in 2009, Thakur \cite{Thakur2009} conjectured that all Thakur MZVs are $\bbF_q$-linearly independent.
This conjecture has been disproved recently by Im--Kim--Ngo Dac \cite[Theorem A]{IKN2026+}.
They showed that  
\begin{align*}
& \zeta_A(q+2, q-1)+2 \zeta_A(3,2 q-2)+\zeta_A(q, 2, q-1)-\zeta_A(1, q, q) \\
-& \zeta_A(1,1,2 q-1)-\zeta_A(1, q-1,1, q)+\zeta_A(q+1,1, q-1)+\zeta_A(2, q-1,1, q-1) \\
+& \zeta_A(2, q, q-1)+\zeta_A(2,1,2 q-2)+\zeta_A(q, 1,1, q-1)-\zeta_A(1,1, q-1, q)=0.
\end{align*}
This renders the problem of determining all $\bbF_q$-linear relations among Thakur's MZVs an interesting one.

It turns out that the closely related notion of Carlitz multiple polylogarithm values (abbreviated as CMPLVs) is very convenient in the study of Thakur's MZVs.
In \cite{Chang2014}, Chang introduced the Carlitz multiple polylogarithm function: for a tuple $\fs =(s_1,\ldots,s_k) \in \cI_{>0}$, 
\[
\operatorname{Li}_{\fs}(z_1,z_2,\ldots,z_k) \coloneq \sum_{\substack{n_1>n_2>\cdots>n_k \\ n_1,n_2,\ldots,n_k \in \bbZ_{\geqslant 0}}}  \frac{ z_1^{q^{n_1}} z_2^{q^{n_2}} \cdots z_k^{q^{n_k}} }{\ell_{n_1}^{s_1} \ell_{n_2}^{s_2} \cdots \ell_{n_k}^{s_k}}
\]
for multiple variables $z_1,\ldots,z_k$ in a certain range, where $\ell_n = \prod_{i=1}^{n} (\theta-\theta^{q^i})$ for $n \in \bbZ_+$ and $\ell_0=1$.
We will only use the values of CMPLs at $(z_1,\ldots,z_k)=(1,\ldots,1)$.
Therefore, it is convenient to define
\[
\LiA(\fs) \coloneq \sum_{\substack{n_1>n_2>\cdots>n_k \\ n_1,n_2,\ldots,n_k \in \bbZ_{\ge 0}}}  \frac{ 1 }{\ell_{n_1}^{s_1} \ell_{n_2}^{s_2} \cdots \ell_{n_k}^{s_k}}.
\]
We also define the Carlitz multiple polylogarithm star values (abbreviated as CMPLSVs)
\[
\LiA^{\star}(\fs) \coloneq \sum_{\substack{n_1\geqslant n_2 \geqslant \cdots \geqslant n_k \\ n_1,n_2,\ldots,n_k \in \bbZ_{\geqslant 0}}}  \frac{ 1 }{\ell_{n_1}^{s_1} \ell_{n_2}^{s_2} \cdots \ell_{n_k}^{s_k}}.
\]
By convention, $\LiA(\emptyset)=\LiAS(\emptyset) \coloneq 1$.

Define $\cZ_{w}^{(\bbF_q)}$ to be the $\bbF_q$-linear subspace of $K_\infty$ spanned by Thakur's MZVs $\zeta_A(\mathfrak{s})$ of weight $w$.
The following theorem by Im--Kim--Ngo Dac connects Thakur's MZVs and CMPLVs.

\begin{theorem}[{Im--Kim--Ngo Dac \cite[Theorem 3.1]{IKN2026+}}]\label{thm_MZV=CMPLV}
We have
\[
\cZ_w^{(\bbF_q)} = \Span_{\bbF_q} \left\{ \LiA(\fs) \mid \fs \in \cI_w \right\}.
\]
\end{theorem}

\begin{remark}\label{rmk_1.6}
Theorem \ref{thm_MZV=CMPLV} also holds if we replace $\bbF_q$ by $\bbF_p$, and there exists an algorithm by Im--Kim--Ngo Dac to express each Thakur's MZV as an $\bbF_p$-linear combination of CMPLVs, and vice versa.

Moreover, we have $\zeta_A(\fs)=\LiA(\fs)$ for any $\fs \in [q]^{\bullet}$ (see \cite[\S 3.1--3.3]{Thakur2009}). 
In particular, Theorem \ref{thm_CCMIKLNP} implies that 
\[
\left\{ \theta^{k} \LiA(\ft) ~\Big|~ k \in \bbZ_{\ge 0}, \ft \in \BTTw \right\}
\]
is an $\bbF_q$-linearly independent set.
\end{remark}

\subsection{Main results in this paper}
We completely solve the problem of $\bbF_q$-linear relations among Thakur's MZVs.
Our main results are the following Theorems \ref{thm_Z} and \ref{thm_R}.
Let $\Rw$ be the space of $\bbF_q$-linear relations among CMPLVs of weight $w$.

\begin{theorem}\label{thm_Z}
For any weight $w \ge 1$, we have $\dim_{\bbF_q} \cZ_w^{(\bbF_q)} = d_w$, where the sequence $\{d_w\}_{w \geqslant 1}$ is defined by the generating function 
\[
\sum_{w=1}^{\infty} d_w x^w = \frac{x(1-x^q)(1-2x+x^q)}{(1-2x+x^{q+1})^2}.
\]
Moreover, $\left\{ \LiA^{\star}(\fs) ~\big|~ \fs \in \cB_w  \right\}$ is an $\bbF_q$-basis of $\cZ_w^{(\bbF_q)}$, where $\cB_w = \cB \cap \cI_w$ and
\[
\cB = [q]^{\bullet} \bigsqcup \left\{ (\fh,m,\ft) \in \cI ~\Big|~ \fh,\ft \in [q]^{\bullet}, m \in [2q] \setminus [q] \right\}.
\]
\end{theorem}

\begin{theorem}[Theorem \ref{thm_R2}]\label{thm_R}
The $\bbF_q$-linear space $\Rw$ is generated by quadruple-carry relations of weight $w$.
\end{theorem}

We will define the quadruple-carry relations in Section \ref{subsec_HL}.
As our result indicates, these quadruple-carry relations are the most natural $\mathbb{F}_q$-linear relations among CMPLVs.
Moreover, they are analogous in spirit to the double-shuffle relations among classical MZVs, since they arise from certain double-counting arguments. 

\begin{remark}
Note that Theorem \ref{thm_Z} implies 
\[
\dim_{\bbF_q} \cZ_w^{(\bbF_q)} \sim c w \lambda^w
\]
as $w \to \infty$ with $q$ fixed, for some constant $c>0$ depending only on $q$, where $\lambda$ is the unique real root in the interval $(1,2)$ of the polynomial $x^q-x^{q-1}-x^{q-2}-\cdots-1$.
Comparing this with the corollary of Theorem \ref{thm_CCMIKLNP}:
\[
\dim_{K} \cZ_w^{(K)} \sim c^\prime \lambda^w,
\]
we find that $\dim_{\bbF_q} \cZ_w^{(\bbF_q)}$ is not exponentially larger than $\dim_{K} \cZ_w^{(K)}$, which is quite unexpected.
\end{remark}

\begin{remark}
Our proof still works (with necessary minor modifications) if we replace $\bbF_q$ by any subfield $L$ of $\bbF_q$.
That is, if we denote by $\cZ_w^{(L)}$ the $L$-linear subspace of $K_{\infty}$ spanned by all $\zeta_A(\fs)$ of weight $w$, then for any weight $w \geqslant 1$, 
\[
\dim_{L} \cZ_w^{(L)} = d_w
\]
is the same as $\dim_{\bbF_q} \cZ_w^{(\bbF_q)}$. 
\end{remark}

\noindent\textbf{Statement on contribution.} The first and last authors contributed primarily to this work; they discovered the quadruple-carry relations. 
The individual contribution proportions for all four authors are detailed as follows: J. Hu: 40\%, H. Huang: 5\%, L. Lai: 10\%, K. Li: 45\%.

\noindent\textbf{Statement on AI use.} All the main ideas were developed by the authors.
The paper was written by the authors.
ChatGPT (GPT-5.6 Sol, OpenAI) was used in the following ways:
\begin{itemize}
\item to write a {\Cpp} program to compute examples of weight-$w$ $\bbF_q$-linear relations for some $q \leqslant 7$ and $w \leqslant 20$;
\item to discover the six-family-$\bQ$ identity (Lemma \ref{lem_6Q}), which is a crucial ingredient for the dimension upper bound of the space $\Qw$ generated by quadruple-carry relations (Proposition \ref{prop_other_half});
\item to construct a more elegant basis of $\cZ_w^{(\bbF_q)}$ (Theorem \ref{thm_Z2}) than the first basis constructed by the authors. 
\item to verify and simplify other proofs made by the authors;
\end{itemize}
All mathematical statements, proofs, computations, and verifications contributed by ChatGPT (GPT-5.6 Sol, OpenAI) mentioned above were subsequently examined, rigorously checked, and rewritten by the authors, who take full responsibility for the results.

\begin{acknowledgements}
This work was completed during the Algebra and Number Theory Summer School held at Peking University, organized by Shouwu Zhang and Liang Xiao, from July 13 to August 21, 2026.
We are grateful to Peking University for its outstanding research environment and support.
L.L. is supported by Research Foundation for Scholars of Xiamen University (Grant No. X2450218).
\end{acknowledgements}

\subsection{Structure of this paper and the notation list}

The structure of this paper is as follows.
In Section \ref{sec_HL}, we recall carry relations and introduce quadruple-carry relations.
Sections \ref{sec_Ind}--\ref{sec_Rw=Qw} prove that every $\bbF_q$-linear relation among CMPLVs is generated by quadruple-carry relations.
In Sections \ref{sec_7} and \ref{sec_8}, we use the method of minimal tuples and two identities among quadruple-carry relations to determine the dimension of $\Qw$.
Finally, Section \ref{sec_9} identifies a basis of $\cZ_w^{(\bbF_q)}$ and proves Theorem \ref{thm_Z}.

\vspace{3 mm}
\noindent\textbf{Notation}: 
\begin{align*}
\bbF_q &= \text{the finite field with $q$ elements, for $q$ a power of a prime number $p$}; \\
A &=\bbF_q[\theta] = \text{the polynomial ring in the variable $\theta$ over $\bbF_q$}; \\
A_{+} &= \text{the set of monic polynomials in $A$}; \\
\bbZ_{+} &= \text{the set of positive integers}; \\
D_1 &= \theta^q - \theta \in A_{+}; \\
K &= \bbF_q(\theta) = \text{the fraction field of $A$}; \\
K_{\infty} &= \bbF_q(\!(1/\theta)\!) = \text{the completion of $K$ with respect to the place at infinity}; \\
\cI &= \bbZ_{+}^{\bullet} = \{\emptyset\} \sqcup \bigsqcup_{k=1}^{\infty} \bbZ_{+}^{k} = \text{the set of tuples};\\
\cI_w &= \left\{ \fs \in \cI \mid \wt(\fs) = w \right\} = \text{the set of tuples of weight $w$}; \\ 
\zeta_A(\fs) &= \text{Thakur's multiple zeta value associated to the tuple $\fs \in \cI$}; \\
\LiA(\fs) &= \text{the Carlitz multiple polylogarithm associated to the tuple $\fs \in \cI$}; \\
\cZ_w^{(L)} &= \Span_{L} \left\{ \zeta_A(\fs) \mid \fs \in \cI_w \right\} = \text{the $L$-linear subspace of $K_{\infty}$ spanned by Thakur's}\\
&\quad\ \text{multiple zeta values $\zeta_A(\fs)$ of weight $w$, for $L = \bbF_q$ and $L=K$}; \\
\mathscr{S}_w^{(R)} &= \left\{ \bS=\sum_{\fs \in \cI_w} c_{\fs} \cdot \fs ~\Big|~ c_{\fs} \in R \right\} = \text{the $R$-module consisting of formal sums of }\\
&\quad\ \text{tuples of weight $w$, for $R = \bbF_q$, $R=\bbF_q[D_1]$, and $R=K$}; \mathscr{S}^{(R)} = \bigoplus_{w=0}^{\infty} \mathscr{S}_w^{(R)}; \\
\mathscr{R}_w^{(R)} &= \left\{ \bR \in \mathscr{S}_w^{(R)} \mid \LiA(\bR) = 0 \right\} = \text{the $R$-module consisting of $R$-linear relations} \\ 
&\quad	\text{among Carlitz multiple polylogarithm values $\LiA(\fs)$ of weight $w$}; \\
\prec &= \text{the (depth, lex)-order}; \triangleleft = \text{the (depth, $\overleftarrow{\text{lex}}$)-order}; \\
s(\bS) &= \text{the minimal tuple appearing in a formal sum $\bS \neq \bzero$ with respect to the}\\ &\quad\ \text{(depth,lex)-order};\\ 
\brho\left( \fh, \boxed{q+r}, \ft \right) &= \text{a carry relation, see Lemma \ref{lem_gIKN}}; \bQ\left( \fh,\boxed{q+r_1},\fm,\boxed{q+r_2},\ft \right) = \text{a} \\
&\quad\	\text{quadruple-carry relation, see Lemma \ref{lem_HL}}; \\
\IKNw &= \text{the set of Im--Kim--Ngo Dac relations of weight $w$}; \\
\gIKNw &= \text{the set of carry relations of weight $w$}; \\ 
\HLw &= \text{the set of quadruple-carry relations of weight $w$}; \\
\Qw &= \text{the $\bbF_q$-linear subspace of $\Rw$ spanned by $\HLw$}; \\
[n] &= \{1,2,\ldots,n\} \text{ for } n \in \bbZ_{+}; [0] = \emptyset; \\
[q]^{\bullet} &= \left\{ \fs=(s_1,\ldots,s_k) \in \bbZ_{+}^{\bullet} \mid k \in \bbZ_{\ge 0}, s_i \leqslant q \text{ for all } i \in [k]  \right\} = \left\{ \emptyset \right\} \sqcup \bigsqcup_{k=1}^{\infty} [q]^k;\\
* &= \text{the stuffle product} = \text{the harmonic product}; \\
\widetilde{*} &= \text{the ``sticky product''}; \\
\sfE, \sfC, \sfL &= \text{linear operators over $\mathscr{S}^{(\bbF_q)}$, see Definition \ref{def_ECL}}; \\ 
\sfZ &= \text{a linear isomorphism of $\Sw$, see Definition \ref{def_Z_operator}}.
\end{align*}

\section{Carry relations and quadruple-carry relations}\label{sec_HL}

For a subring $R$ of $K$ (we only use the cases $R=\bbF_q$, $R=\bbF_q[D_1]$, and $R=K$), define 
\[
\mathscr{S}_w^{(R)} = \left\{ \bS=\sum_{\fs \in \cI_w} c_{\fs} \cdot \fs ~\Big|~ c_{\fs} \in R \right\}
\]
to be the $R$-module of formal sums of tuples of weight $w$.
Define $\mathscr{S}^{(R)}=\bigoplus_{w=0}^{\infty} \mathscr{S}_w^{(R)}$.
Two formal sums $\bS=\sum_{\fs \in \cI_w} c_{\fs} \cdot \fs$ and $\bS^\prime =\sum_{\fs \in \cI_w} c_{\fs}^\prime \cdot \fs$ are equal if and only if $c_{\fs}=c_{\fs}^\prime$ for every $\fs \in \cI_w$.

For any formal sum $\bS=\sum_{\fs \in \cI_w} c_{\fs} \cdot \fs \in \mathscr{S}_w^{(R)}$, define the $\LiA$-evaluation 
\[
\LiA(\bS) = \LiA\left(\sum_{\fs \in \cI_w} c_{\fs} \cdot \fs\right) \coloneq \sum_{\fs \in \cI_w} c_{\fs}\LiA(\fs) \in K_{\infty}.
\]
A formal sum $\bR$ is called a \emph{relation} if $\LiA(\bR) = 0$.
Define
\[
\mathscr{R}_w^{(R)} = \left\{ \bR \in \mathscr{S}_w^{(R)} ~\Big|~ \LiA(\bR) = 0 \right\}.
\]

The stuffle product $*$ is a $K$-bilinear map $\mathscr{S}^{(K)} \times \mathscr{S}^{(K)} \longrightarrow \mathscr{S}^{(K)}$ such that
\begin{align*}
& \emptyset * \bS = \bS * \emptyset = \bS, \\
& \fs * \ft = \left(s_1, \widetilde{\fs} * \ft\right) + \left(t_1, \fs * \widetilde{\ft}\right) + \left(s_1+t_1, \widetilde{\fs} * \widetilde{\ft}\right)
\end{align*}
for any $\bS \in \mathscr{S}^{(K)}$ and any $\fs = \left(s_1, \widetilde{\fs}\right)$, $\ft = \left(t_1,\widetilde{\ft}\right) \in \cI_{>0}$, where $s_1,t_1 \in \bbZ_+$ and $\widetilde{\fs},\widetilde{\ft} \in \cI$.

Throughout this paper, $D_1 = \theta^q -\theta \in A_+$.

\subsection{Carry relations}\label{subsec_IKN}
Our starting point is the general version of the $\bbF_q[D_1]$-linear relations found by Im--Kim--Ngo Dac \cite{IKN2026+}, restated in the following Lemma \ref{lem_gIKN}.
We refer to these $\bbF_q[D_1]$-linear relations as \emph{carry relations}, or \emph{general Im--Kim--Ngo Dac relations} (abbreviated as gIKN relations), as they resemble carrying over on the components of tuples. In a certain sense, they also provide a ``carry'' (=lift) from $\mathbb{F}_q$-level tuples to $D_1$-level tuples. 
For a tuple $\fs = (\widetilde{\fs},s) \in \cI_{>0}$ (where $s \in \bbZ_+$ and $\widetilde{\fs} \in \cI$) and a positive integer $a$, define $\fs^{+a} = (\widetilde{\fs},s+a)$.
The notation $\fs^{+}$ is the abbreviation of $\fs^{+1}$.
By convention, $\emptyset^{+a}=(a)$.

\begin{lemma}[carry relations, or gIKN relations]\label{lem_gIKN}
For any $\fh,\ft \in \cI$ and $r \in \bbZ_{\ge 0}$ satisfying either $r>0$, or $r=0$ and $\ft=\emptyset$ (with the convention that, in the latter case, $(r,\ft)=(0,\emptyset)$ is identified with $\emptyset$), define the following formal sum of tuples:
\begin{align}
\brho\left( \fh, \boxed{q+r}, \ft \right) &\coloneq \left( \fh, q+r, \ft \right) + \mathbbm{1}_{r \neq 0} \cdot \left( \fh, q, r, \ft \right) \notag\\
&\quad+D_1 \cdot \mathbbm{1}_{\fh \neq \emptyset} \cdot \left( \fh^{+}, (q-1)*(r,\ft) \right) + D_1 \cdot \left( \fh, 1, (q-1)*(r,\ft) \right), \label{def_rho}
\end{align}
where $D_1 = \theta^q-\theta \in A_+$.
Then $\boldsymbol{\rho}\left( \fh, \boxed{q+r}, \ft \right) \in \mathscr{R}_w^{(\mathbb{F}_q[D_1])}$, where $w=\operatorname{wt}\left( \fh, q+r, \ft \right)$. 
That is, 
\[
\LiA\left( \boldsymbol{\rho}\left( \fh, \boxed{q+r}, \ft \right) \right) = 0.
\]
\end{lemma}

\begin{proof}
If we restrict $\fh \in [q]^{\bullet}$, then this lemma is exactly \cite[Proposition 4.3]{IKN2026+}.
However, their proof works the same for any $\fh \in \cI$.
\end{proof}

\begin{example}
The carry relation $\boldsymbol{\rho}\left( \boxed{q} \right) = (q) + D_1 \cdot (1,q-1)$ corresponds to the fundamental relation of Thakur \cite[Lemma 5]{Thakur2009}, since $\LiA\left(\boldsymbol{\rho}\left( \boxed{q} \right)\right) = 0$ reads as $\LiA(q)+D_1\LiA(1,q-1)=0$.
\end{example}

\subsection{The (depth, lex)-order}
As already noted by Im--Kim--Ngo Dac in \cite{IKN2026+}, the carry relations behave well with respect to the (depth, lex)-order on $\cI$.
\begin{definition}
The (depth, lex)-order $\prec$ on $\cI$ is defined as follows.
For any two tuples $\fs_1, \fs_2 \in \cI$, define $\fs_1 \prec \fs_2$ if and only if
\begin{itemize}
\item either $\operatorname{dep}(\fs_1) < \operatorname{dep}(\fs_2)$,
\item or $\operatorname{dep}(\fs_1) = \operatorname{dep}(\fs_2)$ and $\fs_1$ is lexicographically smaller than $\fs_2$.
\end{itemize}
Note that $\prec$ is a well-order on $\cI$. 
For any formal sum of tuples $\boldsymbol{S}=\sum_{\fs\in\cI} c_{\fs} \cdot \fs \in \mathscr{S}^{(K)}$ with $\boldsymbol{S} \neq \boldsymbol{0}$, define $s(\boldsymbol{S})$ to be the minimal tuple appearing in $\boldsymbol{S}$. That is,
\[
s(\boldsymbol{S}) = \min_{\prec} \left\{ \fs \mid c_{\fs} \neq 0 \right\}.
\]
\end{definition}  

\begin{remark}\label{rmk_IKN}
For any $\fh,\ft \in \cI$ and $r \in \bbZ_{\geqslant 0}$ such that either $r>0$, or $r=0$ and $\ft=\emptyset$, we have
\[
\boldsymbol{\rho}\left( \fh, \boxed{q+r}, \ft \right) = \left( \fh, q+r, \ft \right) + \sum_{\substack{\fs \in \cI_w \\ \fs \succ \left( \fh, q+r, \ft \right)}} c_{\fs} \cdot \fs
\]
for some $c_{\fs} \in \mathbb{F}_p[D_1] \subset \mathbb{F}_q[D_1]$, where $w = \operatorname{wt}\left( \fh, q+r, \ft \right)$.
In particular, 
\[
s\left( \boldsymbol{\rho}\left( \fh, \boxed{q+r}, \ft \right) \right) = \left( \fh, q+r, \ft \right).
\]
\end{remark}

\subsection{Quadruple-carry relations}\label{subsec_HL}
In this subsection, we introduce our key observation in this paper. 
Roughly speaking, a double-counting argument that applies the carry relations four times to a sum of four suitably related CMPLVs yields an $\mathbb{F}_q$-linear relation among CMPLVs.
Therefore, we call them \emph{quadruple-carry relations}. 
We also refer to such $\mathbb{F}_q$-linear relations as \emph{Hu--Li relations} (abbreviated as HL relations), as they were discovered by the first and last authors.
These quadruple-carry relations are the most natural $\mathbb{F}_q$-linear relations among CMPLVs.
In spirit, they are analogous to the double-shuffle relations, which are the most natural $\mathbb{Q}$-linear relations among classical MZVs.

\begin{lemma}[quadruple-carry relations, or HL relations]\label{lem_HL}
For any $\fh,\fm,\ft \in \cI$, $r_1 \in \bbZ_+$, and $r_2 \in \bbZ_{\geqslant 0}$ satisfying either $r_2>0$, or $r_2=0$ and $\ft=\emptyset$ (with the convention that, in the latter case, $(r_2,\ft)=(0,\emptyset)$ is identified with $\emptyset$), define the following formal sum of tuples:
\begin{align}
& \bQ\left( \fh,\boxed{q+r_1},\fm,\boxed{q+r_2},\ft \right) \notag\\
\coloneq & \left( \fh, (q+r_1,\fm)^{+}, (q-1)*(r_2,\ft) \right) + \left( \fh, q+r_1, \fm, 1, (q-1)*(r_2,\ft) \right) \notag\\
&+ \left( \fh, q, (r_1,\fm)^{+}, (q-1)*(r_2,\ft) \right) + \left( \fh, q, r_1, \fm, 1, (q-1)*(r_2,\ft) \right) \notag\\
&- \mathbbm{1}_{\fh \neq \emptyset} \cdot \left( \fh^{+},(q-1)*(r_1,\fm,q+r_2,\ft) \right) - \left( \fh,1, (q-1)*(r_1,\fm,q+r_2,\ft) \right) \notag\\
&- \mathbbm{1}_{\fh \neq \emptyset, r_2 \neq 0} \cdot \left( \fh^{+},(q-1)*(r_1,\fm,q,r_2,\ft) \right) - \mathbbm{1}_{r_2 \neq 0} \cdot \left( \fh,1,(q-1)*(r_1,\fm,q,r_2,\ft) \right). \label{def_Q}
\end{align}
Then, 
\begin{align}
&D_1 \cdot \bQ\left( \fh,\boxed{q+r_1},\fm,\boxed{q+r_2},\ft \right) = \brho\left( \fh, q+r_1, \fm, \boxed{q+r_2}, \ft \right) + \brho\left( \fh, q, r_1, \fm, \boxed{q+r_2}, \ft \right) \notag\\
&\qquad\qquad\qquad\qquad\qquad -\brho\left( \fh, \boxed{q+r_1}, \fm, q+r_2, \ft \right) -\mathbbm{1}_{r_2 \neq 0} \cdot \brho\left( \fh, \boxed{q+r_1}, \fm, q, r_2, \ft \right). \label{eqn_DQ=4rho}
\end{align}
In particular, $\bQ\left( \fh,\boxed{q+r_1},\fm,\boxed{q+r_2},\ft \right) \in \mathscr{R}_w^{(\mathbb{F}_q)}$, where $w=\operatorname{wt}\left( \fh,q+r_1,\fm,q+r_2,\ft \right)$. 
That is,
\[
\LiA\left(\bQ\left( \fh,\boxed{q+r_1},\fm,\boxed{q+r_2},\ft \right)\right) = 0.
\]
\end{lemma}

\begin{proof}
By the definition of carry relations (see Equation \eqref{def_rho} of Lemma \ref{lem_gIKN}), we have 
\begin{align}
&\brho\left( \fh, q+r_1, \fm, \boxed{q+r_2}, \ft \right) = \left( \fh, q+r_1, \fm, q+r_2, \ft \right) + \mathbbm{1}_{r_2 \neq 0} \cdot \left( \fh, q+r_1, \fm, q, r_2, \ft \right) \notag\\
&\qquad +D_1 \cdot \Big( \left( \fh, (q+r_1,\fm)^{+}, (q-1)*(r_2,\ft) \right) + \left( \fh, q+r_1, \fm, 1, (q-1)*(r_2,\ft) \right) \Big), \label{eqn_2.1} \\
&\brho\left( \fh, q, r_1, \fm, \boxed{q+r_2}, \ft \right) = \left( \fh, q, r_1, \fm, q+r_2, \ft \right) + \mathbbm{1}_{r_2 \neq 0} \cdot \left( \fh, q, r_1, \fm, q, r_2, \ft \right) \notag\\
&\qquad +D_1 \cdot \Big( \left( \fh, q, (r_1,\fm)^{+}, (q-1)*(r_2,\ft) \right) + \left( \fh, q, r_1, \fm, 1, (q-1)*(r_2,\ft) \right) \Big), \label{eqn_2.2} \\
&\brho\left( \fh, \boxed{q+r_1}, \fm, q+r_2, \ft \right) = \left( \fh, q+r_1, \fm, q+r_2, \ft \right) + \left( \fh, q, r_1, \fm, q+r_2, \ft \right) \notag\\
&\qquad +D_1 \cdot \Big( \mathbbm{1}_{\fh \neq \emptyset} \cdot \left( \fh^{+},(q-1)*(r_1,\fm,q+r_2,\ft) \right) + \left( \fh,1, (q-1)*(r_1,\fm,q+r_2,\ft) \right) \Big), \label{eqn_2.3} \\
&\mathbbm{1}_{r_2 \neq 0} \cdot \brho\left( \fh, \boxed{q+r_1}, \fm, q, r_2, \ft \right) = \mathbbm{1}_{r_2 \neq 0}\cdot\left( \fh, q+r_1, \fm, q, r_2, \ft \right) + \mathbbm{1}_{r_2 \neq 0}\cdot\left( \fh, q, r_1, \fm, q, r_2, \ft \right) \notag\\
&\quad +D_1 \cdot \Big( \mathbbm{1}_{\fh \neq \emptyset, r_2 \neq 0} \cdot \left( \fh^{+},(q-1)*(r_1,\fm,q,r_2,\ft) \right) + \mathbbm{1}_{r_2 \neq 0} \cdot \left( \fh,1,(q-1)*(r_1,\fm,q,r_2,\ft) \right) \Big). \label{eqn_2.4}
\end{align}
Taking the sum of \eqref{eqn_2.1} and \eqref{eqn_2.2}, and then subtracting the sum of \eqref{eqn_2.3} and \eqref{eqn_2.4}, the $\bbF_q$-level tuples cancel out; we obtain Identity \eqref{eqn_DQ=4rho}.

Taking the $\LiA$-evaluation of \eqref{eqn_DQ=4rho} and using Lemma \ref{lem_gIKN}, we deduce that 
\[
D_1 \cdot \LiA\left( \bQ\left( \fh,\boxed{q+r_1},\fm,\boxed{q+r_2},\ft \right) \right)=0.
\]
Therefore, $\LiA\left( \bQ\left( \fh,\boxed{q+r_1},\fm,\boxed{q+r_2},\ft \right) \right)=0$.
Note that the coefficients appearing in the formal sum $\bQ\left( \fh,\boxed{q+r_1},\fm,\boxed{q+r_2},\ft \right)$ all belong to $\mathbb{F}_p \subset \mathbb{F}_q$, so
\[
\bQ\left( \fh,\boxed{q+r_1},\fm,\boxed{q+r_2},\ft \right) \in \mathscr{R}_w^{(\mathbb{F}_q)}.
\]
\end{proof}

\begin{example}
The quadruple-carry relation
\[
\bQ\left( \boxed{q+1}, \boxed{q} \right) = (q+2,q-1) + (q+1,1,q-1) + (q,2,q-1) + (q,1,1,q-1) - (1,(q-1)*(1,q))
\]
corresponds to the first non-trivial $\mathbb{F}_q$-linear relation among CMPLVs found by Im--Kim--Ngo Dac in \cite{IKN2026+}. 
That is,
\begin{align*}
&\LiA(q+2,q-1) + \LiA(q+1,1,q-1) + \LiA(q,2,q-1) + \LiA(q,1,1,q-1) \\
-&\LiA(1,q-1,1,q) - \LiA(1,q,q) -\LiA(1,1,q-1,q) - \LiA(1,1,2q-1) -\LiA(1,1,q,q-1) = 0.
\end{align*}
\end{example}

Lemma \ref{lem_HL} is the source of all relations considered in this paper; its leading-term shape is recorded in Remark \ref{rmk_HL} below.

\begin{remark}\label{rmk_HL}
Note that the relation $\bQ=\bQ\left( \fh,\boxed{q+r_1},\fm,\boxed{q+r_2},\ft \right)$ can be written as 
\[
\bQ=\left( \fh, (q+r_1,\fm)^{+}, (q-1)\widetilde{*}(r_2,\ft) \right) + \sum c_{\fs} \cdot \fs,
\]
where $\widetilde{*}$ is the sticky product (for $a \in \bbZ_{+}$ and $(s_1,\ldots,s_k) \in \cI$, the sticky product $(a)\widetilde{*}(s_1,\ldots,s_k)$ is defined by $\sum_{i=1}^{k} (s_1,\ldots,s_{i-1},s_i+a,s_{i+1},\ldots,s_k)$, with the convention $(a)\widetilde{*}\emptyset=(a)$), and that the sum is taken over finitely many tuples $\fs$ satisfying
\[
\fs \succ (\fh^{+},\{1\}^{\operatorname{dep}(\fm)+\operatorname{dep}(\ft)+2}),
\] 
with $c_{\fs} \in \mathbb{F}_p \subset \mathbb{F}_q$ (we take the convention $\emptyset^{+}=(1)$).
We may therefore simply say
\[
\bQ = \left( \fh, (q+r_1,\fm)^{+}, (q-1)\widetilde{*}(r_2,\ft) \right) + \text{``higher tuples''}.
\]
\end{remark}

We introduce the following notation, which is frequently used in this paper.
\begin{definition}\label{def_Qw}
For any weight $w \ge 1$, define $\mathcal{Q}_w$ to be the set of all quadruple-carry relations of weight $w$. 
That is,
\begin{align*}
\HLw \coloneq &\left\{ \bQ\left( \fh,\boxed{q+r_1},\fm,\boxed{q+r_2},\ft \right) ~\Big|~ \fh,\fm,\ft \in \cI, r_1,r_2 \in \bbZ_+, \wt\left( \fh, q+r_1,\fm, q+r_2,\ft \right) = w \right\} \\
&\bigsqcup \left\{ \bQ\left( \fh,\boxed{q+r_1},\fm,\boxed{q} \right) ~\Big|~ \fh,\fm \in \cI, r_1 \in \bbZ_+, \wt\left( \fh, q+r_1,\fm, q \right) = w \right\} \subset \mathscr{R}_w^{(\mathbb{F}_q)}.
\end{align*}
Define
\[
\mathscr{Q}_w^{(\mathbb{F}_q)} \coloneq \operatorname{Span}_{\mathbb{F}_q} \mathcal{Q}_w \subset \mathscr{R}_w^{(\mathbb{F}_q)}.
\]
\end{definition}

\section{The operator $\Ind$ and the set $\gInd$}\label{sec_Ind}

\subsection{The operator $\Ind$ of Im--Kim--Ngo Dac}\label{subsec_Ind}

Recall the carry relations (or gIKN relations) $\brho\left( \fh, \boxed{q+r}, \ft \right)$ defined in Lemma \ref{lem_gIKN}.
It is convenient to introduce the following two sets of relations. 

\begin{definition}
For any weight $w$, define
\begin{align*}
\gIKNw &\coloneq \left\{ \brho\left( \fh, \boxed{q+r}, \ft \right) ~\Big|~ \fh,\ft \in \cI, r \in \bbZ_{+}, \wt\left( \fh, q+r, \ft \right) = w \right\} \\
&\qquad \bigsqcup \left\{ \brho\left( \fh, \boxed{q} \right) ~\Big|~ \fh \in \cI, \wt\left( \fh, q \right) = w \right\} \subset \RwFD, \\
\IKNw &\coloneq \left\{ \brho\left( \fh, \boxed{q+r}, \ft \right) ~\Big|~ \fh \in [q]^{\bullet},\ft \in \cI, r \in \bbZ_{+}, \wt\left( \fh, q+r, \ft \right)=w \right\} \\
&\qquad \bigsqcup \left\{ \brho\left( \fh, \boxed{q} \right) ~\Big|~ \fh \in [q]^{\bullet}, \wt\left( \fh, q \right)=w \right\} \subset \gIKNw.
\end{align*}
\end{definition}
Thus, $\gIKNw$ is the set of all carry relations of weight $w$.
We refer to elements in $\IKNw$ as \emph{special carry relations}, or \emph{Im--Kim--Ngo Dac relations} (abbreviated as IKN relations); 
they are precisely the relations studied by Im--Kim--Ngo Dac in \cite{IKN2026+}.

The IKN relations are special in the following sense.
For any $\fs \in \cI_w \setminus \BTTw$, where $\BTTw$ is the set of Todd--Thakur tuples defined in Theorem \ref{thm_CCMIKLNP}, there exists a unique choice of $\fh$, $\ft$, and $r$ such that
\begin{itemize}
\item $\fs=(\fh,q+r,\ft)$, where $\fh \in [q]^{\bullet}$, $\ft \in \cI$, and $r \in \bbZ_{\geqslant 0}$; and
\item either $r >0$, or $r=0$ and $\ft = \emptyset$.
\end{itemize}
Thus, there is a unique $\brho_{\fs} \in \IKNw$, where $w=\wt(\fs)$, corresponding to $\fs$; that is,
\[
\brho_{\fs} = \brho\left( \fh, \boxed{q+r}, \ft \right) \in \IKNw.
\]
We have 
\[
\IKNw = \left\{ \brho_{\fs} \mid \fs \in \cI_w \setminus \BTTw \right\}.
\]
Moreover, Theorem \ref{thm_RwK} can be restated as follows: The set $\IKNw$ is a basis of the $K$-linear space $\RwK$. 

The following two lemmas were observed by Im--Kim--Ngo Dac in \cite{IKN2026+}, which follow directly from Lemma \ref{lem_gIKN} and Remark \ref{rmk_1.6}.

\begin{lemma}\label{lem_nonTTtoTT}
For any $\fs \in \cI_w\setminus\BTTw$, there exist $c_\ft \in \bbF_p[D_1] \subset \bbF_q[D_1]$ such that
\[
\LiA(\fs) = \sum_{\ft \in \BTTw,\ \ft \succ \fs} c_\ft \LiA(\ft).
\]
\end{lemma}

\begin{lemma}\label{lem_min_tuple_not_Todd}
For any $\bsigma \in \RwK \setminus\{\bzero\}$, we have
\[
s(\boldsymbol{\sigma}) \notin \BTTw.
\] 
\end{lemma}

Now, we describe the algorithm introduced by Im--Kim--Ngo Dac in \cite{IKN2026+}, which is used later to define the operator $\Ind$.

Let $N = \#\left( \cI_w \setminus \BTTw \right)$ and list the tuples in $\cI_w \setminus \BTTw$ in ascending order with respect to the (depth, lex)-order:
\[
\cI_w \setminus \BTTw = \left\{ \fs_1 \prec \fs_2 \prec \cdots \prec \fs_N \right\}.
\]
For any relation $\bR \in \Rw \setminus \{\bzero\}$, we have $s(\bR)=\fs_{i_0}$ for some $i_0 \in [N]$ by Lemma \ref{lem_min_tuple_not_Todd}.
We conduct the following algorithm:
\begin{itemize}
\item Initially, define $\bsigma_{i_0} = \bR = \sum_{i=i_0}^{N} a_i^{(i_0)} \cdot \fs_i + \sum_{\ft \in \BTTw} a_{\ft}^{(i_0)} \cdot \ft$, where $a_i^{(i_0)},a_{\ft}^{(i_0)} \in \bbF_q$.

\item For $i_0 \leqslant k \leqslant N$, suppose that $\bsigma_{k}$ is already defined and takes the form
\[
\bsigma_{k} = \sum_{i=k}^{N} a_{i}^{(k)} \cdot \fs_i + \sum_{\ft \in \BTTw} a_{\ft}^{(k)} \cdot \ft +\quad D_1 \sum_{\fs \in \cI_w,\ \fs \succ s(\bR)} b_{\fs}^{(k)} \cdot \fs,
\]
where $a_i^{(k)},a_{\ft}^{(k)},b_{\fs}^{(k)} \in \bbF_q$.
Define $\bsigma_{k+1} = \bsigma_{k} - a_{k}^{(k)}\brho_{\fs_k}$.
Then, by Remark \ref{rmk_IKN}, $\bsigma_{k+1}$ must take the form
\[
\bsigma_{k+1} = \sum_{i=k+1}^{N} a_{i}^{(k+1)} \cdot \fs_i + \sum_{\ft \in \BTTw} a_{\ft}^{(k+1)} \cdot \ft +\quad D_1 \sum_{\fs \in \cI_w,\ \fs \succ s(\bR)} b_{\fs}^{(k+1)} \cdot \fs,
\]
where $a_i^{(k+1)},a_{\ft}^{(k+1)},b_{\fs}^{(k+1)} \in \bbF_q$.

\item At the end, $\bsigma_{N+1}$ has been defined and takes the form
\begin{equation}\label{eqn_end_status1}
\bsigma_{N+1} =  \sum_{\ft \in \BTTw} a_{\ft}^{(N+1)} \cdot \ft + D_1 \sum_{\fs \in \cI_w,\ \fs \succ s(\bR)} b_{\fs}^{(N+1)} \cdot \fs,
\end{equation}
where $a_{\ft}^{(N+1)},b_{\fs}^{(N+1)} \in \bbF_q$.
Note that we have $\LiA(\bsigma_k)=0$ for any $i_0 \leqslant k \leqslant N+1$ by Lemma \ref{lem_gIKN}.
By Remark \ref{rmk_1.6}, the set $\{ D_1^{k}\LiA(\ft) \mid \ft \in \BTTw, k \in \bbZ_{\ge 0}\}$ is $\bbF_q$-linearly independent.
Therefore, after using Lemma \ref{lem_nonTTtoTT} to express the $\LiA$-evaluation of the right-hand side of \eqref{eqn_end_status1} as an $\bbF_q[D_1]$-linear combination of $\LiA(\ft)$ ($\ft \in \BTTw$), we see that $a_{\ft}^{(N+1)} = 0$ for any $\ft \in \BTTw$. 
Thus, the end status  $\bsigma_{N+1}$ takes the form
\begin{equation}\label{eqn_end_status}
\bsigma_{N+1} = D_1 \sum_{\fs \in \cI_w,\ \fs \succ s(\bR)} b_{\fs}^{(N+1)} \cdot \fs.
\end{equation}
\end{itemize}

\begin{definition}
The operator $\Ind: \Rw \rightarrow \Rw$ is defined as follows.
For $\bzero \in \Rw$, define $\Ind(\bzero)=\bzero$.
For $\bR \in \Rw \setminus \{\bzero\}$, let $\bsigma_{N+1}$ be the end status \eqref{eqn_end_status} of the above algorithm by Im--Kim--Ngo Dac, and define 
\[
\Ind(\bR) = \sum_{\fs \in \cI_w,\ \fs \succ s(\bR)} b_{\fs}^{(N+1)} \cdot \fs \in \Rw.
\]
\end{definition}

The following lemma follows immediately from the definition of $\Ind$.

\begin{lemma}\label{lem_Ind_nilp}
For any $\bR \in \Rw \setminus \{\bzero\}$, either $\Ind(\bR) = \bzero$ or $s(\Ind(\bR)) \succ s(\bR)$.
\end{lemma}

\subsection{The set $\gInd(\bR)$}
We generalize the ideas of Im--Kim--Ngo Dac on the operator $\Ind \colon \Rw \rightarrow \Rw$.
Roughly speaking, we replace special carry relations by general carry relations in their algorithm. 
This leads to the following definition of the set $\gInd(\bR)$.

\begin{definition}
For any relation $\bR \in \Rw$, define the set $\gInd(\bR)$ by
\[
\gInd(\bR) \coloneq \left\{ \bR^\prime \in \Rw ~\Big|~ D_1 \bR^\prime = \bR + \sum_{i} c_i \brho_i \text{ for finitely many } \brho_i \in \gIKNw \text{ with } c_i \in \bbF_q \right\}.
\]
\end{definition}

\begin{lemma}\label{lem_easy}
For any relation $\bR \in \Rw$, we have 
\[
\Ind(\bR) \in \gInd(\bR) \quad\text{and}\quad \gInd(\bR) = \Ind(\bR) + \gInd(\bzero) \subset \Rw.
\]
\end{lemma}

\begin{proof}
It follows directly from the definition.
\end{proof}

By Lemma \ref{lem_easy}, to understand the sets $\gInd(\bR)$, it remains to determine $\gInd(\bzero)$. The next section proves that this is exactly the space spanned by quadruple-carry relations.

\section{A characterization of $\Qw$ in terms of $\gInd(\bzero)$}

\begin{definition}\label{def_LU}
For any $\fh,\ft \in \cI$ and $r \in \bbZ_{\ge 0}$ satisfying either $r>0$, or $r=0$ and $\ft=\emptyset$ (with the convention that, in the latter case, $(r,\ft)=(0,\emptyset)$ is identified with $\emptyset$), define
\begin{align*}
\bU\left( \fh,\boxed{q+r},\ft \right) \coloneq \mathbbm{1}_{\fh \neq \emptyset} \cdot \left( \fh^{+}, (q-1)*(r,\ft) \right) + \left( \fh, 1, (q-1)*(r,\ft) \right) \in \mathscr{S}^{(\bbF_q)}.
\end{align*}
\end{definition}
Note that
\[
\brho\left( \fh, \boxed{q+r}, \ft \right) = \left( \fh, q+r, \ft \right) + \mathbbm{1}_{r \neq 0} \cdot \left( \fh, q, r, \ft \right) + D_1 \cdot \bU\left( \fh,\boxed{q+r},\ft \right).
\]
Therefore, $\bU\left( \fh, \boxed{q+r}, \ft \right)$ is the formal sum of $D_1$-level tuples of $\brho\left( \fh, \boxed{q+r}, \ft \right)$.

The following Proposition \ref{prop_gInd0} is crucial.

\begin{proposition}\label{prop_gInd0}
For $\bzero \in \Rw$, we have
\[
\gInd(\bzero) = \Qw.
\]
\end{proposition}

\begin{proof}
For any $\bQ \in \HLw$, by Lemma \ref{lem_HL}, there exist four or three $\brho_i \in \gIKNw$ such that $D_1 \bQ = \sum_{i} (\pm \brho_i)$.
Therefore, for any $\bQ \in \Qw$, we have $D_1 \bQ = \sum_{i} c_i\brho_i$ for finitely many $\brho_i \in \gIKNw$ with $c_i \in \bbF_q$, so $\bQ \in \gInd(\bzero)$, and hence $\Qw \subset \gInd(\bzero)$.

In the following, we prove that any $\bR \in \gInd(\bzero)$ must belong to $\Qw$. 
Since $\bR \in \gInd(\bzero)$, there exist $c\left( \fh, \boxed{q+r}, \ft \right) \in \bbF_q$ such that
\begin{align}
D_1 \bR &= \sum  c\left( \fh, \boxed{q+r}, \ft \right) \brho \left( \fh, \boxed{q+r}, \ft \right) \notag\\
&= \sum_{\substack{\fh \in \cI,\ft \in \cI, r \in \bbZ_{+} \\ \wt(\fh,q+r,\ft) = w}} c\left( \fh, \boxed{q+r}, \ft \right) \brho \left( \fh, \boxed{q+r}, \ft \right) + \sum_{\substack{\fh^\prime \in \cI \\ \wt(\fh^\prime,q) = w}} c\left( \fh^\prime, \boxed{q} \right) \brho \left( \fh^\prime, \boxed{q} \right) \label{eqn_4.00} \\
&= \sum_{\substack{\fh \in \cI,\ft \in \cI, r \in \bbZ_{+} \\ \wt(\fh,q+r,\ft) = w}} c\left( \fh, \boxed{q+r}, \ft \right) \cdot \left( \left( \fh, q+r, \ft \right) + \left( \fh, q, r, \ft \right) + D_1 \bU\left( \fh,\boxed{q+r},\ft \right) \right) \notag\\
&\quad +\sum_{\substack{\fh^\prime \in \cI \\ \wt(\fh^\prime,q) = w}} c\left( \fh^\prime, \boxed{q} \right) \cdot \left( \left( \fh^\prime, q \right) + D_1 \bU\left( \fh^\prime,\boxed{q} \right) \right). \label{eqn_4.0}
\end{align}
Since Equation \eqref{eqn_4.0} is an identity between formal sums of tuples, we must have
\begin{align}
\sum_{\substack{\fh \in \cI,\ft \in \cI, r \in \bbZ_{+} \\ \wt(\fh,q+r,\ft) = w}} c\left( \fh, \boxed{q+r}, \ft \right) \cdot \Big( \left( \fh, q+r, \ft \right) + \left( \fh, q, r, \ft \right) \Big) + \sum_{\substack{\fh^\prime \in \cI \\ \wt(\fh^\prime,q) = w}} c\left( \fh^\prime, \boxed{q} \right) \cdot \left( \fh^\prime, q \right) &= \bzero, \label{eqn_4.2} \\
\sum_{\substack{\fh \in \cI,\ft \in \cI, r \in \bbZ_{+} \\ \wt(\fh,q+r,\ft) = w}} c\left( \fh, \boxed{q+r}, \ft \right)  \bU\left( \fh,\boxed{q+r},\ft \right) + \sum_{\substack{\fh^\prime \in \cI \\ \wt(\fh^\prime,q) = w}} c\left( \fh^\prime, \boxed{q} \right) \bU\left( \fh^\prime,\boxed{q} \right) &= \bR. \label{eqn_4.3}
\end{align}

Denote by $\bc$ the sequence consisting of all the coefficients $c\left( \fh, \boxed{q+r}, \ft \right)$ and $c\left( \fh^\prime, \boxed{q} \right)$, and define
\begin{align*}
\operatorname{supp}_1 \bc &\coloneq \left\{ \left( \fh, q+r, \ft \right) ~\Big|~ c\left( \fh, \boxed{q+r}, \ft \right) \neq 0, \fh \in \cI,\ft \in \cI, r \in \bbZ_{+}, \wt(\fh,q+r,\ft) = w \right\}, \\
\operatorname{supp}_2 \bc &\coloneq \left\{ \left( \fh, q,r, \ft \right) ~\Big|~ c\left( \fh, \boxed{q+r}, \ft \right) \neq 0, \fh \in \cI,\ft \in \cI, r \in \bbZ_{+}, \wt(\fh,q+r,\ft) = w \right\}, \\
\operatorname{supp}_3 \bc &\coloneq \left\{ (\fh^\prime,q) ~\Big|~ c\left( \fh^\prime, \boxed{q} \right) \neq 0, \fh^\prime \in \cI, \wt(\fh^\prime,q) = w \right\}, \\
\operatorname{supp} \bc &\coloneq \bigcup_{i=1}^{3} \operatorname{supp}_i \bc \subset \cI_w.  
\end{align*}
If $\operatorname{supp} \bc = \emptyset$, then $\bc = \bzero$ and hence $\bR = \bzero$ by \eqref{eqn_4.3}. 
In this case, there is nothing to prove.
If $\operatorname{supp} \bc \neq \emptyset$, take $\fs = \min_{\prec} (\operatorname{supp} \bc) \in \cI_w$.
We must have $\fs \notin \operatorname{supp}_2 \bc$; otherwise, there would exist $\fh, \ft \in \cI$ and $r \in \bbZ_+$ such that $\fs=(\fh,q,r,\ft)$ with $c\left( \fh, \boxed{q+r}, \ft \right) \neq 0$, and then $(\fh, q+r, \ft) \in \operatorname{supp}_1 \bc$ and $(\fh, q+r, \ft) \prec \fs$, contradicting the minimality of $\fs$. 
Since \eqref{eqn_4.2} is an identity between formal sums of tuples, comparing the coefficients of $\fs$ in \eqref{eqn_4.2} and using $\fs \notin \operatorname{supp}_2 \bc$, we obtain 
\begin{equation}\label{eqn_4.4}
\sum_{\substack{\fh \in \cI,\ft \in \cI, r \in \bbZ_{+} \\ (\fh,q+r,\ft) = \fs}} c\left( \fh, \boxed{q+r}, \ft \right) + \sum_{\substack{\fh^\prime \in \cI \\ (\fh^\prime,q) = \fs}} c\left( \fh^\prime, \boxed{q} \right) = 0.
\end{equation}
Note that there is at most one $\fh^\prime \in \cI$ satisfying $(\fh^\prime, q) = \fs$. 
We distinguish two cases.

\textbf{Case 1}: $\fs \notin \operatorname{supp}_3 \bc$. 
In this case, we must have $\fs \in \operatorname{supp}_1 \bc$, and Equation \eqref{eqn_4.4} becomes 
\begin{equation}\label{eqn_4.5}
\sum_{\substack{\fh \in \cI,\ft \in \cI, r \in \bbZ_{+} \\ (\fh,q+r,\ft) = \fs}} c\left( \fh, \boxed{q+r}, \ft \right) = 0.
\end{equation}
By \eqref{eqn_4.5} and $\fs \in \operatorname{supp}_1 \bc$, the following finite set of marked tuples
\[
\operatorname{Mark}_{\fs} \coloneq \left\{ \left( \fh, \boxed{q+r}, \ft \right) ~\Big|~ \left( \fh, q+r, \ft \right) = \fs, c\left( \fh, \boxed{q+r}, \ft \right) \neq 0 \right\}
\]
contains at least two elements.
We may therefore suppose 
\[
\operatorname{Mark}_{\fs} = \left\{ \left( \fh_i, \boxed{q+r_i}, \ft_i \right) ~\Big|~ i=0,1,\ldots,m \right\}
\] 
for some $m \in \bbZ_+$, and the marked position in $\left( \fh_0, \boxed{q+r_0}, \ft_0 \right)$ is the leftmost one among all marked positions in $\left( \fh_i, \boxed{q+r_i}, \ft_i \right)$, $i=0,1,\ldots,m$.
Now, for any $i \in [m]$, since 
\[
\left( \fh_i, q+r_i, \ft_i \right) = \fs = \left( \fh_0, q+r_0, \ft_0 \right),
\]
we must have
\[
\fs = \left( \fh_0, q+r_0, \fm_i, q+r_i, \ft_i \right) \quad\text{for some } \fm_i \in \cI.
\]
Notice that Equation \eqref{eqn_4.5} becomes
\begin{equation}\label{eqn_4.8}
\sum_{i=0}^{m} c_i = 0, \quad\text{where } c_i = c\left( \fh_i, \boxed{q+r_i}, \ft_i \right).
\end{equation}
Adding the identity (see Equation \eqref{eqn_DQ=4rho} of Lemma \ref{lem_HL})
\begin{align*}
& -D_1 \sum_{i=1}^{m} c_i \bQ\left( \fh_0, \boxed{q+r_0}, \fm_i, \boxed{q+r_i}, \ft_i \right) \\
=& \sum_{i=1}^{m} c_i \Big( \brho\left( \fh_0, \boxed{q+r_0}, \ft_0 \right) + \brho\left( \fh_0, \boxed{q+r_0}, \fm_i, q, r_i, \ft_i \right) \\
&\qquad - \brho\left( \fh_i, \boxed{q+r_i}, \ft_i \right) - \brho\left( \fh_0, q,r_0, \fm_i, \boxed{q+r_i}, \ft_i \right) \Big)
\end{align*}
to Equation \eqref{eqn_4.00}, and using \eqref{eqn_4.8}, we deduce that 
\begin{align*}
\sum_{\substack{\fh \in \cI,\ft \in \cI, r \in \bbZ_{+} \\ \wt(\fh,q+r,\ft) = w}} c^\prime\left( \fh, \boxed{q+r}, \ft \right) \cdot \Big( \left( \fh, q+r, \ft \right) + \left( \fh, q, r, \ft \right) \Big) + \sum_{\substack{\fh^\prime \in \cI \\ \wt(\fh^\prime,q) = w}} c^\prime\left( \fh^\prime, \boxed{q} \right) \cdot \left( \fh^\prime, q \right) &= \bzero,  \\
\sum_{\substack{\fh \in \cI,\ft \in \cI, r \in \bbZ_{+} \\ \wt(\fh,q+r,\ft) = w}} c^\prime\left( \fh, \boxed{q+r}, \ft \right)  \bU\left( \fh,\boxed{q+r},\ft \right) + \sum_{\substack{\fh^\prime \in \cI \\ \wt(\fh^\prime,q) = w}} c^\prime\left( \fh^\prime, \boxed{q} \right) \bU\left( \fh^\prime,\boxed{q} \right) &= \bR^\prime, 
\end{align*}
where
\[
\bR^\prime = \bR - \sum_{i=1}^{m} c_i \bQ\left( \fh_0, \boxed{q+r_0}, \fm_i, \boxed{q+r_i}, \ft_i \right),
\]
and that the coefficient sequence $\bc^\prime$ is either $\bzero$ or satisfies 
\[
\min(\operatorname{supp} \bc^{\prime}) \succ \min(\operatorname{supp} \bc).
\]

\textbf{Case 2}: $\fs \in \operatorname{supp}_3 \bc$. 
In this case, there exists a unique $\fh^{\prime} \in \cI$ such that $(\fh^\prime,q)=\fs$ and $c\left(\fh^\prime,\boxed{q}\right) \neq 0$.
Equation \eqref{eqn_4.4} becomes 
\begin{equation}\label{eqn_4.9}
\sum_{\substack{\fh \in \cI,\ft \in \cI, r \in \bbZ_{+} \\ (\fh,q+r,\ft) = \fs}} c\left( \fh, \boxed{q+r}, \ft \right) +  c\left( \fh^\prime, \boxed{q} \right) = 0.
\end{equation}
By \eqref{eqn_4.9} and $c\left(\fh^\prime,\boxed{q}\right) \neq 0$, the following finite set of marked tuples
\[
\operatorname{Mark}_{\fs}^\prime \coloneq \left\{ \left( \fh, \boxed{q+r}, \ft \right) ~\Big|~ \left( \fh, q+r, \ft \right) = \fs, c\left( \fh, \boxed{q+r}, \ft \right) \neq 0, r>0 \right\}
\]
is non-empty. 
We may assume
\[
\operatorname{Mark}_{\fs}^\prime = \left\{ \left( \fh_i, \boxed{q+r_i}, \ft_i \right) ~\Big|~ i=1,2,\ldots,m^\prime \right\}
\] 
for some $m^\prime \in \bbZ_+$. 
Now, for any $i \in [m^\prime]$, since 
\[
\left( \fh_i, q+r_i, \ft_i \right) = \fs = \left(\fh^\prime,q\right),
\]
we must have 
\[
\fs = \left( \fh_i, q+r_i, \fm_i, q \right) \quad\text{for some } \fm_i \in \cI.
\]
Notice that Equation \eqref{eqn_4.9} becomes
\begin{equation}\label{eqn_4.10}
\sum_{i=1}^{m^\prime} c_i  +  c\left( \fh^\prime, \boxed{q} \right) = 0, \quad\text{where } c_i = c\left( \fh_i, \boxed{q+r_i}, \ft_i \right).
\end{equation}
Adding the identity
\begin{align*}
& D_1\sum_{i=1}^{m^\prime} c_i\bQ\left( \fh_i,\boxed{q+r_i}, \fm_i, \boxed{q} \right) \\
=& \sum_{i=1}^{m^\prime} c_i \Big( \brho\left( \fh^\prime, \boxed{q} \right) + \brho\left( \fh_i, q,r_i, \fm_i, \boxed{q} \right) -\brho\left( \fh_i, \boxed{q+r_i}, \ft_i \right) \Big)
\end{align*}
to Equation \eqref{eqn_4.00}, and using \eqref{eqn_4.10}, by the same reasoning as in Case 1, we obtain a new 
\[
\bR^\prime = \bR + \sum_{i=1}^{m^\prime} c_i \bQ\left( \fh_i,\boxed{q+r_i}, \fm_i, \boxed{q} \right),
\]
and the corresponding coefficient sequence $\bc^\prime$ is either $\bzero$ or satisfies 
\[
\min(\operatorname{supp} \bc^{\prime}) \succ \min(\operatorname{supp} \bc).
\]

In both cases, we have shown that there exist finitely many $\bQ_i \in \HLw$ with $c_i \in \bbF_q$ such that
\[
\bR^{\prime} = \bR - \sum_i c_i\bQ_i,
\]
and the corresponding coefficient sequence $\bc^\prime$ is either $\bzero$ or satisfies 
\[
\min(\operatorname{supp} \bc^{\prime}) \succ \min(\operatorname{supp} \bc).
\]
If $\bc^\prime = \bzero$, then $\bR^\prime = \bzero$ and hence $\bR \in \Qw$. Otherwise, we repeat the process, and find finitely many $\bQ_i^\prime \in \HLw$ with $c_i^\prime \in \bbF_q$ such that
\[
\bR^{\prime\prime} = \bR^\prime - \sum_i c_i^\prime \bQ_i^\prime,
\]
and the corresponding coefficient sequence $\bc^{\prime\prime}$ is either $\bzero$ or satisfies 
\[
\min(\operatorname{supp} \bc^{\prime\prime}) \succ \min(\operatorname{supp} \bc^\prime).
\]
Since $\cI_w$ is a finite set, we will eventually reduce the coefficient sequence $\bc$ to zero, and thus conclude that $\bR = \sum_i c_i \bQ_i$ is an $\bbF_q$-linear combination of quadruple-carry relations.
This completes the proof of $\gInd(\bzero) = \Qw$. 
\end{proof}

\begin{remark}
The idea behind the proof of Proposition \ref{prop_gInd0} can be interpreted using the language of graph theory. 
Indeed, consider the digraph with vertex set $\cI_w$. 
For two vertices $\fs_1$ and $\fs_2$, draw a weighted directed edge from $\fs_1$ to $\fs_2$ if $\fs_1=(\fh,q+r,\ft)$ and $\fs_2=(\fh,q,r,\ft)$ for some $\fh,\ft,r$, and assign weight $c\left( \fh,\boxed{q+r},\ft \right) \in \bbF_q$ to this directed edge.
Equation \eqref{eqn_DQ=4rho} says that a quadruple-carry relation $\bQ(\fh,\boxed{q+r_1},\fm,\boxed{q+r_2},\ft)$ corresponds to the boundary of a square in this digraph, see Figure \ref{fig:quadruple-carry-square}.
The reduction in the proof subtracts such square boundaries in increasing order of their minimal vertex, thereby cancelling all edges.  
The case $r=0$ and $\ft=\emptyset$ is slightly exceptional; the formal-sum argument above includes it (Case 2) without requiring a separate graph-theoretic convention.
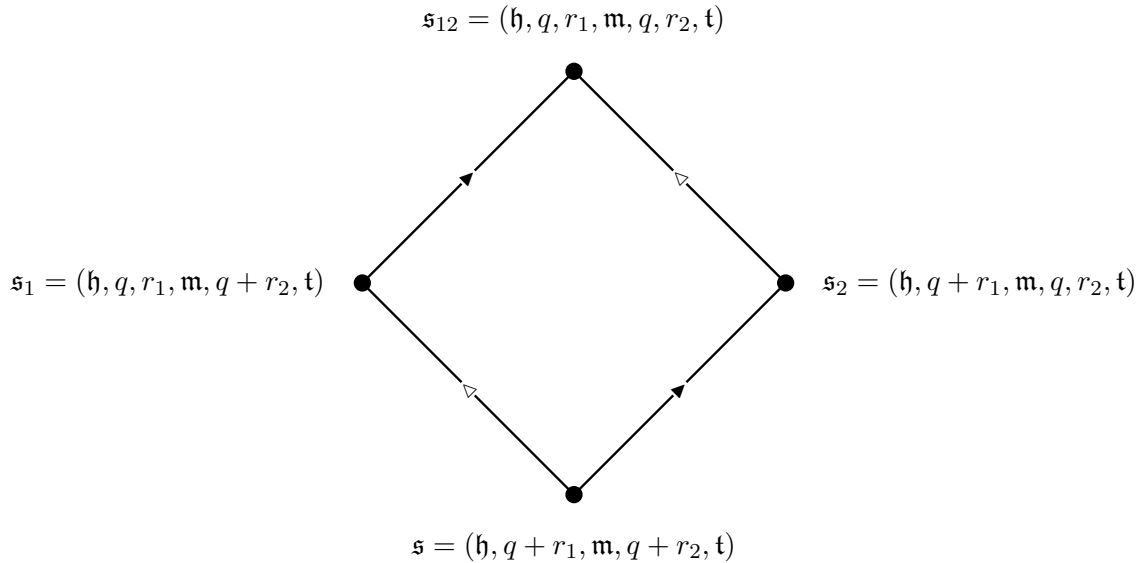
\begin{figure}[htbp]
\centering
\begin{tikzpicture}[
vertex/.style={
circle,
fill=black,
inner sep=2.3pt
},
edge/.style={
line width=0.9pt
},
markstyle/.style={
sloped,
allow upside down,
fill=white,
inner sep=0.4pt
}
]
\coordinate (S12) at (0,2.8);
\coordinate (S1)  at (-2.8,0);
\coordinate (S2)  at (2.8,0);
\coordinate (S)   at (0,-2.8);
\draw[edge] (S) -- (S2)
node[pos=0.5,markstyle]
{\(\scriptstyle\blacktriangleright\)};
\draw[edge] (S12) -- (S2)
node[pos=0.5,markstyle]
{\(\scriptstyle\vartriangleleft\)};
\draw[edge] (S1) -- (S12)
node[pos=0.5,markstyle]
{\(\scriptstyle\blacktriangleright\)};
\draw[edge] (S1) -- (S)
node[pos=0.5,markstyle]
{\(\scriptstyle\vartriangleleft\)};
\node[vertex] at (S12) {};
\node[vertex] at (S1) {};
\node[vertex] at (S2) {};
\node[vertex] at (S) {};
\node[above=10pt] at (S12)
{\small\( \fs_{12}=(\fh,q,r_1,\fm,q,r_2,\ft)\)};
\node[left=10pt] at (S1)
{\small\( \fs_1=(\fh,q,r_1,\fm,q+r_2,\ft)\)};
\node[right=10pt] at (S2)
{\small\( \fs_2=(\fh,q+r_1,\fm,q,r_2,\ft)\)};
\node[below=10pt] at (S)
{\small\( \fs=(\fh,q+r_1,\fm,q+r_2,\ft)\)};
\end{tikzpicture}
\caption{The square boundary corresponding to a quadruple-carry relation. 
In the associated relations, $\blacktriangleright$ denotes the positive sign and $\vartriangleright$ denotes the negative sign.}
\label{fig:quadruple-carry-square}
\end{figure}
\end{remark}

\begin{corollary}\label{cor_coset}
For any relation $\boldsymbol{R} \in \mathscr{R}_w^{(\mathbb{F}_q)}$, 
\[
\gInd(\bR) = \Ind(\bR) + \Qw 
\]
is a coset of $\Qw$ in $\Rw$.
\end{corollary}

\begin{proof}
It follows from Proposition \ref{prop_gInd0} and Lemma \ref{lem_easy}.
\end{proof}

The following notation is convenient in the rest of this paper.
\begin{definition}\label{def_q+}
For any $a \in \bbZ_+$ and any non-empty tuple $\fs=(s_1,\widetilde{\fs}) \in \cI_{>0}$ (where $s_1 \in \bbZ_{+}$ and $\widetilde{\fs} \in \cI$), the notation ${}^{a+}\fs$ is defined by ${}^{a+}\fs=(a+s_1,\widetilde{\fs})$.
The notation $\boxed{{}^{a+}\fs}$ means $(\boxed{a+s_1},\widetilde{\fs})$. 
Take the convention ${}^{a+}\emptyset = (a)$ and ${}^{+}\fs = {}^{1+}\fs$.

For example, ${}^{q+}(1,2)=(q+1,2)$ and ${}^{q+}(\emptyset,2)={}^{q+}(2)=(q+2)$.
	
Both notations are extended linearly to formal sums. 
More generally, whenever a tuple-valued argument of $\brho$, $\bU$, $\bQ$, etc., is a formal sum, the resulting expression is interpreted multilinearly.

For example, $\brho(\fh,\boxed{q+r},\ft)$ can be rewritten as $\brho(\fh,\boxed{{}^{q+}(r,\ft)})$, and $\brho(\fh,\boxed{q+r},\ft_1+\ft_2)$ means $\brho(\fh,\boxed{q+r},\ft_1) + \brho(\fh,\boxed{q+r},\ft_2)$.
\end{definition}

\begin{lemma}\label{lem_gIndQ}
For any $\bQ \in \Qw$, we have 
\[
\gInd(\bQ) = \Qw.
\]
\end{lemma}

\begin{proof}
By Corollary \ref{cor_coset} and the definitions of $\gInd$ and $\Qw$, it suffices to prove that $\bzero \in \gInd(\bQ)$ for any $\bQ \in \HLw$.
We may therefore assume
\[
\bQ = \bQ\left( \fh,\boxed{q+r_1},\fm,\boxed{q+r_2},\ft \right),
\]
where $\fh,\fm,\ft \in \cI$, $r_1 \in \bbZ_+$, $r_2 \in \bbZ_{\geqslant 0}$, and $\wt\left( \fh,q+r_1,\fm, q+r_2,\ft \right) = w$, such that either $r_2>0$, or $r_2=0$ and $\ft=\emptyset$ (with the convention $(r_2,\ft)=\emptyset$).
We claim that
\begin{align}
\bQ -& \brho\left( \fh, \boxed{(q+r_1,\fm)^{+}}, (q-1)*(r_2,\ft) \right) - \brho\left( \fh, \boxed{q+r_1}, \fm, 1, (q-1)*(r_2,\ft) \right) \notag\\
+& \brho\left( \fh, 1, r_1, \fm, \boxed{{}^{q+}((q-1)*(r_2,\ft))} \right) + \brho\left( \fh, 1, (q-1)*(r_1,\fm), \boxed{q+r_2}, \ft \right) \notag\\
+& \mathbbm{1}_{\fh \neq \emptyset}\brho\left( \fh^+, r_1, \fm, \boxed{{}^{q+}((q-1)*(r_2,\ft))} \right) + \mathbbm{1}_{\fh \neq \emptyset}\brho\left( \fh^+, (q-1)*(r_1,\fm), \boxed{q+r_2}, \ft \right) = \bzero. \label{888}
\end{align}
Indeed, using \eqref{def_Q} and \eqref{def_rho}, and collecting terms multilinearly, we have
\[
\text{LHS of \eqref{888}} = \mathbbm{1}_{\fh\neq\emptyset}(\fh^{+},\bS_1) + (\fh,1,\bS_1) + D_1 \mathbbm{1}_{\fh\neq\emptyset}(\fh^{+},\bS_2) + D_1(\fh,1,\bS_2),
\]
where 
\begin{align}
\bS_1 =&~ -(q-1)*(r_1,\fm,q+r_2,\ft) -\mathbbm{1}_{r_2 \neq 0} (q-1)*(r_1,\fm,q,r_2,\ft) \notag\\
&~ +\Big( (r_1,\fm),{}^{q+}\left((q-1)*(r_2,\ft)\right) \Big) + \Big( (r_1,\fm),q,(q-1)*(r_2,\ft) \Big) \notag\\
&~ +\Big( (q-1)*(r_1,\fm), q+r_2,\ft \Big) + \mathbbm{1}_{r_2 \neq 0} \Big( (q-1)*(r_1,\fm), q,r_2,\ft \Big), \label{S1}
\end{align}
and
\begin{align}
\bS_2 =&~ -(q-1)*\Big( (r_1,\fm)^{+},(q-1)*(r_2,\ft) \Big) - (q-1)*\Big( r_1,\fm,1,(q-1)*(r_2,\ft) \Big) \notag\\
&+\Big( (r_1,\fm)^{+},(q-1)*((q-1)*(r_2,\ft)) \Big) +\Big( r_1,\fm,1,(q-1)*((q-1)*(r_2,\ft)) \Big) \notag\\
&+ \Big( \left((q-1)*(r_1,\fm)\right)^{+},(q-1)*(r_2,\ft) \Big) + \Big( (q-1)*(r_1,\fm),1,(q-1)*(r_2,\ft) \Big). \label{S2}
\end{align}
It suffices to show that $\bS_1=\bzero$ and $\bS_2=\bzero$.
We first show $\bS_1 = \bzero$ in the case $r_2 \neq 0$.
By the definition of the stuffle product, we have
\begin{align*}
(q-1)*(r_1,\fm,q+r_2,\ft) =&~ \Big( (q-1)*(r_1,\fm), q+r_2,\ft \Big) + (r_1,\fm,2q-1+r_2,\ft) \\
&~ +\Big( r_1,\fm,q+r_2,(q-1)*\ft \Big), \\
(q-1)*(r_1,\fm,q,r_2,\ft) =&~ \Big( (q-1)*(r_1,\fm), q,r_2,\ft \Big) + (r_1,\fm,2q-1,r_2,\ft) \\
&+~ \Big(r_1,\fm,q,(q-1)*(r_2,\ft)\Big), \\
\Big( (r_1,\fm),{}^{q+}\left((q-1)*(r_2,\ft)\right) \Big) =&~ (r_1,\fm,2q-1,r_2,\ft) + (r_1,\fm,2q-1+r_2,\ft)  \\
&~ +\Big( r_1,\fm,q+r_2,(q-1)*\ft \Big).
\end{align*}
Substituting the above three identities into \eqref{S1}, we obtain $\bS_1=\bzero$ for the case $r_2 \neq 0$. 
The case $r_2=0$ is easier.
And the proof of $\bS_2=\bzero$ is similar.
Thus, \eqref{888} holds true.
It follows that $\bzero \in \gInd(\bQ)$ by definition.
\end{proof}

The remaining ingredient for the proof of $\Rw=\Qw$ is the surjectivity of $\gInd$; this is established in the next section.

\section{A certain surjectivity of $\gInd$}

In this section, our goal is to prove that for any $\bR \in \Rw$, there exists $\bR^\prime \in \Rw$ such that $\bR \in \gInd(\bR^\prime)$.
We start with an easy lemma.

\begin{lemma}\label{lem_min_tuple_imply_dim_lower_bound}
For any field $L$, let $\mathscr{S}^{(L)}$ be the $L$-linear space of formal sums of tuples. 
If a subset $\mathcal{S}$ of $\mathscr{S}^{(L)} \setminus \{\bzero\}$ satisfies $s(\bS_1) \neq s(\bS_2)$ for any two distinct elements $\bS_1 \neq \bS_2$ in $\mathcal{S}$,
then $\mathcal{S}$ is an $L$-linearly independent set in $\mathscr{S}^{(L)}$.
\end{lemma}

\begin{proof}
For any finitely many distinct elements $\bS_1,\ldots,\bS_m$ in $\mathcal{S}$, since the $s(\bS_i)$ are distinct by hypothesis, we may assume $s(\bS_1) \prec \cdots \prec s(\bS_m)$.
For any $(c_1,\ldots,c_m) \in L^{m} \setminus \{\bzero\}$, suppose that $i_0$ is the minimal index such that $c_{i_0} \neq 0$. 
Then $s(\bS_{i_0})$ is the minimal tuple appearing in $\sum_{i=1}^{m} c_i\bS_i$, and therefore $\sum_{i=1}^{m} c_i\bS_i \neq \bzero$.
So $\bS_1,\ldots,\bS_m$ are $L$-linearly independent.
\end{proof}

Recall that $\bU(\fh, \boxed{q+r}, \ft)$ is defined in Definition \ref{def_LU}, which is the formal sum of $D_1$-level tuples of the carry relation $\brho(\fh, \boxed{q+r}, \ft)$.

\begin{definition}\label{def_Uw}
Define
\begin{align*}
\Uw \coloneq \Span_{\bbF_q} \Big( &\left\{ \bU(\fh, \boxed{q+r}, \ft) ~\Big|~ \fh,\ft \in \cI, r \in \bbZ_+, \wt(\fh, q+r, \ft) =w \right\} \\ 
&\bigsqcup \left\{ \bU(\fh, \boxed{q}) ~\Big|~ \fh \in \cI, \wt(\fh,q) = w \right\} \Big) \subset \Sw.
\end{align*}
Define $\mathscr{U}^{(\bbF_q)} = \bigoplus_{w=0}^{\infty} \Uw$.
\end{definition}

Recall that the sequence $\{d_{w}^{(K)}\}_{w \ge 1}$ is defined in Theorem \ref{thm_CCMIKLNP}.

\begin{lemma}\label{lem_dimUw}
For any weight $w \ge 1$, we have 
\[
\dim_{\bbF_q} \Uw \ge 2^{w-1} - d_w^{(K)}.
\]
\end{lemma}

\begin{proof}
Define $\cU = \bigsqcup_{i=1}^{4}\cU^{(i)} \subset \cI$, where
\begin{align*}
\cU^{(1)} &= \left\{ (h,\fm,q+1+a,\ft) ~\Big|~ h \in \bbZ_+, \fm \in [q]^{\bullet}, a \in \bbZ_{\ge 0}, \ft \in \cI  \right\}, \\
\cU^{(2)} &= \left\{ (h,\fm,y,q-1) ~\Big|~ h \in \bbZ_+, \fm \in [q]^{\bullet}, y \in [q]\setminus\{1\}  \right\}, \\
\cU^{(3)} &= \left\{ (h,\fm,q) ~\Big|~ h \in \bbZ_+, \fm \in [q]^{\bullet} \right\}, \\
\cU^{(4)} &= \left\{ (h,q-1) ~\Big|~ h \in \bbZ_+ \right\}.
\end{align*}
For $h \in \bbZ_{+}$, define
\[
h^{-} = \begin{cases}
h-1 &\text{if } h \ge 2, \\
\emptyset &\text{if } h = 1.
\end{cases}
\]

Recall that the notation $\boxed{{}^{q+}\fs}$ is defined in Definition \ref{def_q+}.
For any $\fs_1 = (h,\fm,q+1+a,\ft) \in \cU^{(1)}$ with $\ft \neq \emptyset$, one easily checks that $\fs_1$ is the minimal tuple appearing in
\[
\bU\left( h^{-}, \boxed{{}^{q+}(\fm,2+a)}, \ft \right) - \bU\left( h,\fm,1+a,\boxed{{}^{q+}\ft} \right).
\]
For any $\fs_1 = (h,\fm,q+1+a,\ft) \in \cU^{(1)}$ with $\ft = \emptyset$, it is the minimal tuple appearing in 
\[
\bU\left( h^{-},\boxed{{}^{q+}(\fm,2+a)} \right). 
\]
For any $\fs_2 = (h,\fm,y,q-1) \in \cU^{(2)}$, we have 
\[
s\left( \bU\left( h,\fm,y^{-},\boxed{q} \right) \right) = \fs_2.
\]
For any $\fs_3=(h,\fm,q) \in \cU^{(3)}$, we have
\[
s\left( \bU\left( h^{-},\boxed{{}^{q+}(\fm,1)} \right) \right) = \fs_3.
\]
For any $\fs_4=(h,q-1) \in \cU^{(4)}$, we have
\[
s\left( \bU\left( h^{-},\boxed{q} \right) \right) = \fs_4.
\]
Therefore, for any $\fs \in \cU$, we have constructed an element $\bU_{\fs} \in \mathscr{U}^{(\bbF_q)}$ such that $s(\bU_{\fs}) = \fs$.
By Lemma \ref{lem_min_tuple_imply_dim_lower_bound},
\begin{equation}\label{eqn51}
\dim_{\bbF_q} \Uw \geqslant \#\left( \cU \cap \cI_w \right).
\end{equation}

Clearly, $\cU^{(1)},\ldots,\cU^{(4)}$ are pairwise disjoint.
Moreover, for any fixed $i \in [4]$, the elements in $\cU^{(i)}$ are distinct for different eligible choices of $h$, $\fm$, $a$, $\ft$, $y$.
Denote by $U_w^{(i)} = \#\left(\cU^{(i)} \cap \cI_w \right)$ for each $i \in [4]$. 
We have the generating functions
\begin{align*}
\sum_{w=1}^{\infty} U_w^{(1)} x^w &= \frac{x}{1-x} \cdot \frac{1}{1-x-x^2-\cdots-x^q} \cdot \frac{x^{q+1}}{1-x} \cdot \frac{1-x}{1-2x} = \frac{x^{q+2}}{(1-2x)(1-2x+x^{q+1})}, \\
\sum_{w=1}^{\infty} U_w^{(2)} x^w &= \frac{x}{1-x} \cdot \frac{1}{1-x-x^2-\cdots-x^q} \cdot \frac{x^2-x^{q+1}}{1-x} \cdot x^{q-1} = \frac{x^{q+2}(1-x^{q-1})}{(1-x)(1-2x+x^{q+1})}, \\
\sum_{w=1}^{\infty} U_w^{(3)} x^w &= \frac{x}{1-x} \cdot \frac{1}{1-x-x^2-\cdots-x^q} \cdot x^{q} = \frac{x^{q+1}}{1-2x+x^{q+1}}, \\
\sum_{w=1}^{\infty} U_w^{(4)} x^w &= \frac{x}{1-x} \cdot x^{q-1} = \frac{x^q}{1-x}.
\end{align*} 
Summing up, we obtain
\[
\sum_{w=1}^{\infty} \#\left( \cU \cap \cI_w \right) x^w =  \frac{x^q(1-x)^2}{(1-2x)(1-2x+x^{q+1})} = \frac{x}{1-2x} - \frac{x(1-x^{q-1})}{1-2x+x^{q+1}}.
\]
Therefore, by the definition of $d_w^{(K)}$, we have $\#\left( \cU \cap \cI_w \right) = 2^{w-1} - d_w^{(K)}$.
Combining this with \eqref{eqn51}, we obtain the desired lower bound for $\dim_{\bbF_q} \Uw$. 
\end{proof}

\begin{remark}
We can prove that $\dim_{\bbF_q} \Uw = 2^{w-1} - d_w^{(K)}$. However, Lemma \ref{lem_dimUw} is sufficient for the proof of Proposition \ref{prop_surj} below.
This lower bound on $\dim_{\bbF_q}\Uw$ in Lemma \ref{lem_dimUw} supplies enough independent formal sums $\bU_i$ to represent an arbitrary relation as an element of $\gInd(\bR')$ in the following surjectivity argument.
\end{remark}

\begin{proposition}\label{prop_surj}
For any $\bR \in \Rw$, there exists $\bR^\prime \in \Rw$ such that $\bR \in \gInd(\bR^\prime)$.
\end{proposition}

\begin{proof}
Consider the linear map $\LiA: \mathscr{S}_w^{(K)} \rightarrow K_{\infty}$. 
Its kernel is $\RwK$.
By Theorem \ref{thm_MZV=CMPLV}, its image is $\cZ_w^{(K)}$.
By the rank-nullity theorem and Theorem \ref{thm_CCMIKLNP}, we obtain $\dim_K \RwK = 2^{w-1}-d_w^{(K)}$.

Let $m=2^{w-1}-d_w^{(K)}$. 
By Lemma \ref{lem_dimUw}, there exist $\bU_1,\bU_2,\ldots,\bU_m \in \Uw$ such that they are $\bbF_q$-linearly independent.
By the definition of $\Uw$ (see Definitions \ref{def_LU} and \ref{def_Uw}), there exist $\bF_1,\bF_2,\ldots,\bF_m \in \Sw$ such that $\bF_i+D_1\bU_i \in \Span_{\bbF_q}\gIKNw \subset \RwK$.
Since $\dim_K \RwK = m$, the $m+1$ elements
\[
\bR,\ \bF_1+D_1\bU_1,\ \bF_2+D_1\bU_2,\ \ldots,\ \bF_m+D_1\bU_m
\]
in $\RwK$ are $K$-linearly dependent.
Therefore, for any choice of $m+1$ tuple coordinates, the $(m+1) \times (m+1)$ determinant
\[
\det \left( \bR,\ \bF_1+D_1\bU_1,\ \bF_2+D_1\bU_2,\ \ldots,\ \bF_m+D_1\bU_m \right) = 0 \in K.
\]
On the other hand, by the definition of determinant, we have
\begin{align*}
& \det \left( \bR,\ \bF_1+D_1\bU_1,\ \bF_2+D_1\bU_2,\ \ldots,\ \bF_m+D_1\bU_m \right) \\
=& \det \left( \bR,\ \bU_1,\ \bU_2,\ \ldots,\ \bU_m \right) \cdot D_1^{m} + \text{lower degree terms} \in \bbF_q[D_1].
\end{align*}
Thus, we deduce that
\[
\det \left( \bR,\ \bU_1,\ \bU_2,\ \ldots,\ \bU_m \right) = 0 \in \bbF_q.
\]
This holds for any $m+1$ tuple coordinates, so 
\[
\operatorname{rank} \left( \bR,\ \bU_1,\ \bU_2,\ \ldots,\ \bU_m \right) \le m.
\]
Since $\bU_1,\bU_2,\ldots,\bU_m$ are $\bbF_q$-linearly independent, we must have
\[
\bR \in \Span_{\bbF_q} \left\{ \bU_1,\ \bU_2,\ \ldots,\ \bU_m \right\}.
\]
Therefore, $\bR = \sum_{i=1}^{m} c_i \bU_i$ for some $c_i \in \bbF_q$.
Take $\bR^\prime = -\sum_{i=1}^{m} c_i\bF_i$. Then,
\begin{equation}\label{eqn_5.1}
D_1 \bR = \bR^\prime + \sum_{i=1}^{m} c_i\left( \bF_i+D_1\bU_i \right).
\end{equation}
Taking the $\LiA$-evaluation of \eqref{eqn_5.1}, we obtain $\LiA(\bR^\prime)=0$ and hence $\bR^\prime \in \Rw$.
Since each $\bF_i+D_1\bU_i$ is an $\bbF_q$-linear combination of carry relations, Equation \eqref{eqn_5.1} implies that $\bR \in \gInd(\bR^\prime)$ by definition.
\end{proof}

Having proved the required surjectivity, we can now combine it with the monotonicity of $\Ind$ along the (depth, lex)-order to conclude the main relation theorem.

\section{The operator $\Indbar$ and the proof of $\Rw =\Qw$}\label{sec_Rw=Qw}

We restate Theorem \ref{thm_R} as follows.
\begin{theorem}\label{thm_R2}
For any weight $w \ge 1$, we have
\[
\Rw = \Qw.
\]
In other words, any $\bbF_q$-linear relation among CMPLVs is an $\bbF_q$-linear combination of quadruple-carry relations.
\end{theorem}

\begin{proof}
By Corollary \ref{cor_coset} and Lemma \ref{lem_gIndQ}, there exists a well-defined operator $\Indbar$ on the quotient space $\Rw/\Qw$:
\begin{align*}
\Indbar \colon\ \Rw/\Qw &\longrightarrow \Rw/\Qw \\
\overline{\bR} = \bR + \Qw &\mapsto \gInd(\bR) = \Ind(\bR) + \Qw. 
\end{align*} 
Indeed, by Corollary \ref{cor_coset}, $\gInd(\bR) = \Ind(\bR) + \Qw$ is a coset of $\Qw$ in $\Rw$, and hence can be viewed as an element in the quotient space $\Rw/\Qw$.
If $\bR_1 - \bR_2  \in \Qw$, then there exist finitely many $\brho_{i} \in \gIKNw$ with $c_i \in \bbF_q$ such that $\bR_1 - \bR_2 =\sum_i c_i\brho_i$ by Lemma \ref{lem_gIndQ}, so $\gInd(\bR_1) = \gInd(\bR_2)$ by the definition of $\gInd$.
Therefore, $\Indbar(\overline{\bR})$ is independent of the choice of the representative $\bR$ of $\overline{\bR}$.

By Lemma \ref{lem_Ind_nilp} and the fact $\# \cI_w = 2^{w-1}$ for any $w \ge 1$, we have $\Ind^{2^{w-1}}(\bR) = \bzero$ for any $\bR \in \Rw$. 
Therefore, $\Indbar^{2^{w-1}}(\overline{\bR}) = \overline{\bzero}$ for any $\overline{\bR} \in \Rw/\Qw$.
By Proposition \ref{prop_surj}, $\gInd$ is surjective and hence $\Indbar$ is surjective.
Therefore, we have $\Rw/\Qw = \{ \overline{\bzero} \}$; that is, $\Rw = \Qw$.
\end{proof}

\section{A lower bound for $\dim_{\bbF_q} \Qw$}\label{sec_7}

Recall that the sequence $\{d_w\}_{w \geqslant 1}$ is defined by the generating function 
\[
\sum_{w=1}^{\infty} d_w x^w = \frac{x(1-x^q)(1-2x+x^q)}{(1-2x+x^{q+1})^2},
\]
as stated in Theorem \ref{thm_Z}.
The aim of this section is to prove that $\dim_{\bbF_q} \Qw \ge 2^{w-1} - d_w$.
We use the same strategy as in the proof of Lemma \ref{lem_dimUw}.

\begin{definition}\label{def_M}
Define $\cM = \bigsqcup_{i=1}^{8} \cM^{(i)} \subset \cI$, where 
\begin{align*}
\cM^{(1)} &= \Big\{ (\fh,q+1+a,\fm,2+b,\{1\}^c,q+1+d,\ft,z) ~\Big|~ \fm \in \cI, \fh, \ft \in [q]^{\bullet},  \\
	&\qquad\qquad\qquad\qquad\qquad\qquad\qquad\qquad\qquad\qquad\qquad  a,b,c,d \in \bbZ_{\geqslant 0}, z \in [q-2] \Big\}, \\
	\cM^{(2)} &= \left\{ (\fh,q+1+a,\fm,2+b,\{1\}^c,q+1+d,\ft,1,q-1) ~\Big|~ \fm \in \cI, \fh, \ft \in [q]^{\bullet}, a,b,c,d \in \bbZ_{\geqslant 0} \right\}, \\
	\cM^{(3)} &= \left\{ (\fh,q+1+a,\fm,2+b,\{1\}^c,q+d) ~\Big|~ \fm \in \cI, \fh \in [q]^{\bullet}, a,b,c,d \in \bbZ_{\geqslant 0} \right\}, \\
	\cM^{(4)} &= \left\{ (\fh,q+1+a,\fm,2+b,q-1) ~\Big|~ \fm \in \cI, \fh \in [q]^{\bullet}, a,b \in \bbZ_{\geqslant 0} \right\}, \\
	\cM^{(5)} &= \left\{ (\fh,q+2+a,\{1\}^c,q+1+d,\ft,z) ~\Big|~ \fh,\ft \in [q]^{\bullet}, a,c,d \in \bbZ_{\geqslant 0}, z \in [q-2] \right\}, \\
	\cM^{(6)} &= \left\{ (\fh,q+2+a,\{1\}^c,q+1+d,\ft,1,q-1) ~\Big|~ \fh,\ft \in [q]^{\bullet}, a,c,d \in \bbZ_{\geqslant 0} \right\}, \\
	\cM^{(7)} &= \left\{ (\fh,q+2+a,\{1\}^c,q+d) ~\Big|~ \fh \in [q]^{\bullet}, a,c,d \in \bbZ_{\geqslant 0} \right\}, \\
	\cM^{(8)} &= \left\{ (\fh,q+2+a,q-1) ~\Big|~ \fh \in [q]^{\bullet}, a \in \bbZ_{\geqslant 0} \right\}. 
\end{align*}
Note that if $q=2$, then $[q-2]=\emptyset$ and hence $\cM^{(1)}=\cM^{(5)}=\emptyset$.
For any weight $w$, define $\cM_w = \cM \cap \cI_w$.
\end{definition}

\begin{lemma}\label{lem_cardM}
For any weight $w \geqslant 1$, we have $\# \cM_w = 2^{w-1} - d_w$.
\end{lemma}

\begin{proof}
It is straightforward to check that $\cM^{(i)} \cap \cM^{(j)} = \emptyset$ for $i \neq j$, so $\cM = \bigsqcup_{i=1}^{8} \cM^{(i)}$ is indeed a disjoint union.
Moreover, for any fixed index $i \in [8]$, the tuples in $\cM^{(i)}$ are distinct for different eligible choices of $\fh,\fm,\ft$, $a,b,c,d$, and $z$.

For any $i \in [8]$ and $w \in \bbZ_+$, denote by $\cM^{(i)}_w = \cM^{(i)} \bigcap \cI_w$ and
\[
M_w^{(i)} = \#\cM^{(i)}_w.
\]
Clearly, we have the following generating functions:
\begin{align*}
\sum_{w=1}^{\infty} M^{(1)}_w x^w =&\ \frac{1}{1-x-x^2-\cdots-x^q} \cdot \frac{x^{q+1}}{1-x} \cdot \frac{1-x}{1-2x} \cdot \frac{x^2}{1-x} \cdot \frac{1}{1-x} \cdot \frac{x^{q+1}}{1-x} \\
&\quad\times \frac{1}{1-x-x^2-\cdots-x^q}\cdot\frac{x-x^{q-1}}{1-x} = \frac{x^{2q+5}(1-x^{q-2})}{(1-x)^2(1-2x)(1-2x+x^{q+1})^2},\\
\sum_{w=1}^{\infty} M^{(2)}_w x^w =&\ \frac{1}{1-x-x^2-\cdots-x^q} \cdot \frac{x^{q+1}}{1-x} \cdot \frac{1-x}{1-2x} \cdot \frac{x^2}{1-x} \cdot \frac{1}{1-x} \cdot \frac{x^{q+1}}{1-x} \\
&\quad\times \frac{1}{1-x-x^2-\cdots-x^q}\cdot x^q = \frac{x^{3q+4}}{(1-x)(1-2x)(1-2x+x^{q+1})^2}, \\
\sum_{w=1}^{\infty} M^{(3)}_w x^w =&\ \frac{1}{1-x-x^2-\cdots-x^q} \cdot \frac{x^{q+1}}{1-x} \cdot \frac{1-x}{1-2x} \cdot \frac{x^2}{1-x} \cdot \frac{1}{1-x} \cdot \frac{x^q}{1-x} \\
=&\ \frac{x^{2q+3}}{(1-x)^2(1-2x)(1-2x+x^{q+1})}, \\
\sum_{w=1}^{\infty} M^{(4)}_w x^w =&\ \frac{1}{1-x-x^2-\cdots-x^q} \cdot \frac{x^{q+1}}{1-x} \cdot \frac{1-x}{1-2x} \cdot \frac{x^2}{1-x} \cdot x^{q-1} = \frac{x^{2q+2}}{(1-2x)(1-2x+x^{q+1})}, \\
\sum_{w=1}^{\infty} M^{(5)}_w x^w =&\ \frac{1}{1-x-x^2-\cdots-x^q} \cdot \frac{x^{q+2}}{1-x} \cdot \frac{1}{1-x} \cdot \frac{x^{q+1}}{1-x} \cdot \frac{1}{1-x-x^2-\cdots-x^q} \cdot \frac{x-x^{q-1}}{1-x} \\
=&\ \frac{x^{2q+4}(1-x^{q-2})}{(1-x)^2(1-2x+x^{q+1})^2}, \\
\sum_{w=1}^{\infty} M^{(6)}_w x^w =&\ \frac{1}{1-x-x^2-\cdots-x^q} \cdot \frac{x^{q+2}}{1-x} \cdot \frac{1}{1-x} \cdot \frac{x^{q+1}}{1-x} \cdot \frac{1}{1-x-x^2-\cdots-x^q} \cdot x^q \\
=&\ \frac{x^{3q+3}}{(1-x)(1-2x+x^{q+1})^2}, \\
\sum_{w=1}^{\infty} M^{(7)}_w x^w =&\ \frac{1}{1-x-x^2-\cdots-x^q} \cdot \frac{x^{q+2}}{1-x} \cdot \frac{1}{1-x} \cdot \frac{x^q}{1-x} = \frac{x^{2q+2}}{(1-x)^2(1-2x+x^{q+1})}, \\
\sum_{w=1}^{\infty} M^{(8)}_w x^w =&\ \frac{1}{1-x-x^2-\cdots-x^q} \cdot \frac{x^{q+2}}{1-x} \cdot x^{q-1} = \frac{x^{2q+1}}{1-2x+x^{q+1}}.
\end{align*}
Summing up, we obtain
\[
\sum_{w=1}^{\infty} \left(\sum_{i=1}^{8} M^{(i)}_w\right) x^w = \frac{x^{2q+1}(1-x)^2}{(1-2x)(1-2x+x^{q+1})^2}= \frac{x}{1-2x} - \frac{x(1-x^q)(1-2x+x^q)}{(1-2x+x^{q+1})^2}.
\]
Thus, by the definition of the sequence $\{d_w\}_{w \geqslant 1}$, we have $\#\cM_w = \sum_{i=1}^{8} M^{(i)}_w = 2^{w-1}-d_w$.
\end{proof}

\begin{proposition}\label{prop_half}
For any $\fs \in \cM_w$, there exists an element $\bQ_\fs \in \Qw$ such that $s(\bQ_\fs) = \fs$.
In particular, for any weight $w \ge 1$, we have 
\[
\dim_{\mathbb{F}_q} \mathscr{Q}_w^{(\mathbb{F}_q)} \geqslant 2^{w-1} - d_w.
\]
\end{proposition}

\begin{proof}
Recall that the notation $\boxed{{}^{q+}\fs}$ is defined in Definition \ref{def_q+}.

For any $\fs_1 = (\fh,q+1+a,\fm,2+b,\{1\}^c,q+1+d,\ft,z) \in \cM^{(1)}$, by Lemma \ref{lem_HL} and Remark \ref{rmk_HL}, we have
\begin{align*}
& \bQ\left( \fh,\boxed{q+1+a},\fm, 1+b, \boxed{{}^{q+}(\{1\}^{c}, 2+d)}, \ft,z \right) \\
=&\ \left( \fh,q+1+a,\fm, 2+b, (q-1)\widetilde{*}(\{1\}^{c}, 2+d, \ft,z) \right) + \text{``higher tuples''}, \\
& \bQ\left( \fh,\boxed{q+1+a},\fm,2+b,\{1\}^c, 1+d, \boxed{{}^{q+}(\ft, z)} \right) \\
=&\ \left( \fh,q+1+a,\fm,2+b,\{1\}^c, 2+d, (q-1)\widetilde{*}(\ft, z)\right)  + \text{``higher tuples''}.
\end{align*}
Taking the difference of the two expressions, we obtain
\begin{align*}
& \bQ\left( \fh,\boxed{q+1+a},\fm, 1+b, \boxed{{}^{q+}(\{1\}^{c}, 2+d)}, \ft,z \right) \\
-& \bQ\left( \fh,\boxed{q+1+a},\fm,2+b,\{1\}^c, 1+d, \boxed{{}^{q+}(\ft, z)} \right) \\
=&\ \left( \fh,q+1+a,\fm, 2+b, (q-1)\widetilde{*}(\{1\}^{c}, 2+d), \ft,z \right) + \text{``higher tuples''}.
\end{align*}
Therefore, $\fs_1 = (\fh,q+1+a,\fm,2+b,\{1\}^c,q+1+d,\ft,z) \in \cM^{(1)}$ is the minimal tuple appearing in the relation
\begin{align*}
& \bQ\left( \fh,\boxed{q+1+a},\fm, 1+b, \boxed{{}^{q+}(\{1\}^{c}, 2+d)}, \ft,z \right) \\
- & \bQ\left( \fh,\boxed{q+1+a},\fm,2+b,\{1\}^c, 1+d, \boxed{{}^{q+}(\ft, z)} \right).
\end{align*}

All other cases are similar. 
Indeed, any $\fs_2 = (\fh,q+1+a,\fm,2+b,\{1\}^c,q+1+d,\ft,1,q-1) \in \cM^{(2)}$ is the minimal tuple appearing in the relation
\begin{align*}
& \bQ\left( \fh,\boxed{q+1+a},\fm, 1+b, \boxed{{}^{q+}(\{1\}^{c}, 2+d)}, \ft,1,q-1 \right) \\
- & \bQ\left( \fh,\boxed{q+1+a},\fm,2+b,\{1\}^c, 1+d, \boxed{{}^{q+}(\ft, 1)},q-1 \right).
\end{align*}
Any $\fs_3 = (\fh,q+1+a,\fm,2+b,\{1\}^c,q+d) \in \cM^{(3)}$ is the minimal tuple appearing in the relation
\begin{align*}
& \bQ\left( \fh,\boxed{q+1+a},\fm, 1+b, \boxed{{}^{q+}(\{1\}^{c}, 1+d)} \right).
\end{align*}
Any $\fs_4 = (\fh,q+1+a,\fm,2+b,q-1) \in \cM^{(4)}$ is the minimal tuple appearing in the relation
\begin{align*}
& \bQ\left( \fh,\boxed{q+1+a},\fm, 1+b, \boxed{q} \right).
\end{align*}
Any $\fs_5 = (\fh,q+2+a,\{1\}^c,q+1+d,\ft,z) \in \cM^{(5)}$ is the minimal tuple appearing in the relation
\begin{align*}
& \bQ\left( \fh,\boxed{q+1+a}, \boxed{{}^{q+}(\{1\}^{c}, 2+d)}, \ft,z \right) \\
- & \bQ\left( \fh,\boxed{q+2+a},\{1\}^c, 1+d, \boxed{{}^{q+}(\ft, z)} \right).
\end{align*}
Any $\fs_6 = (\fh,q+2+a,\{1\}^c,q+1+d,\ft,1,q-1) \in \cM^{(6)}$ is the minimal tuple appearing in the relation
\begin{align*}
& \bQ\left( \fh,\boxed{q+1+a}, \boxed{{}^{q+}(\{1\}^{c}, 2+d)}, \ft,1,q-1 \right) \\
- & \bQ\left( \fh,\boxed{q+2+a},\{1\}^c, 1+d, \boxed{{}^{q+}(\ft, 1)},q-1 \right).
\end{align*}
Any $\fs_7 = (\fh,q+2+a,\{1\}^c,q+d) \in \cM^{(7)}$ is the minimal tuple appearing in the relation
\begin{align*}
& \bQ\left( \fh,\boxed{q+1+a}, \boxed{{}^{q+}(\{1\}^{c}, 1+d)} \right).
\end{align*}
Finally, any $\fs_8 = (\fh,q+2+a,q-1) \in \cM^{(8)}$ is the minimal tuple appearing in the relation
\begin{align*}
& \bQ\left( \fh,\boxed{q+1+a}, \boxed{q} \right).
\end{align*}

We have explicitly constructed a $\bQ_\fs \in \Qw$ such that $s(\bQ_\fs)=\fs$ for each $\fs \in \cM_w$.
By Lemmas \ref{lem_min_tuple_imply_dim_lower_bound} and \ref{lem_cardM}, 
\[
\dim_{\mathbb{F}_q} \Qw \geqslant \#\cM_w = 2^{w-1}-d_w.
\]
\end{proof}

\section{An upper bound for $\dim_{\bbF_q} \Qw$}\label{sec_8}

\subsection{The cube identity}

In this subsection, we first prove the cube identity for quadruple-carry relations (Lemma \ref{lem_cube}). 
Then, we use the cube identity to initially normalize quadruple-carry relations.
Recall that the notation ${}^{q+}\fs$ is defined in Definition \ref{def_q+}.

\begin{lemma}[The cube identity]\label{lem_cube}
For any $\fh, \fw \in \cI$ and $\fu, \fv \in \cI_{>0}$, we have
\begin{align*}
&\ \bQ\left( \fh, {}^{q+}\fu, \boxed{{}^{q+}\fv}, \boxed{{}^{q+}\fw} \right) -\bQ\left( \fh, \boxed{{}^{q+}\fu},{}^{q+}\fv,\boxed{{}^{q+}\fw} \right) + \bQ\left( \fh, \boxed{{}^{q+}\fu}, \boxed{{}^{q+}\fv}, {}^{q+}\fw \right) \\
+&\ \bQ\left( \fh, q,\fu, \boxed{{}^{q+}\fv}, \boxed{{}^{q+}\fw} \right) -\bQ\left( \fh, \boxed{{}^{q+}\fu},q,\fv,\boxed{{}^{q+}\fw} \right) + \mathbbm{1}_{\fw \neq \emptyset} \bQ\left( \fh, \boxed{{}^{q+}\fu}, \boxed{{}^{q+}\fv}, q,\fw \right)= \bzero.
\end{align*}
\end{lemma}

\begin{proof}
Multiply by $D_1$ and use \eqref{eqn_DQ=4rho} to express each $D_1 \bQ$ as $\sum_i (\pm \brho_i)$. 
Then all the carry relations $\brho_i$ cancel out.
\end{proof}

\begin{remark}
We can understand the cube identity by the following graph in Figure \ref{fig:cube-identity}.
\begin{figure}[h]
\centering
\begin{tikzpicture}[
scale=1,
markstyle/.style={
sloped,
allow upside down,
fill=white,
inner sep=0.3pt
},
edge/.style={
line width=1.2pt
}
]
\clip (-4,-1) rectangle (4,7);
\coordinate (s)    at (0,0);
\coordinate (s1)   at (-2.8,2);
\coordinate (s2)   at (0,2);
\coordinate (s3)   at (2.8,2);
\coordinate (s12)  at (-2.8,4.2);
\coordinate (s13)  at (0,4.2);
\coordinate (s23)  at (2.8,4.2);
\coordinate (s123) at (0,6.2);
\draw[edge] (s) -- (s1)
node[pos=0.38,markstyle] {\(\scriptstyle\vartriangleright\)}
node[pos=0.62,markstyle] {\(\scriptstyle\vartriangleright\vartriangleright\)};
\draw[edge] (s2) -- (s12)
node[pos=0.38,markstyle] {\(\scriptstyle\vartriangleright\)}
node[pos=0.62,markstyle] {\(\scriptstyle\vartriangleright\vartriangleright\)};
\draw[edge] (s3) -- (s13)
node[pos=0.38,markstyle] {\(\scriptstyle\vartriangleright\)}
node[pos=0.62,markstyle] {\(\scriptstyle\vartriangleright\vartriangleright\)};
\draw[edge] (s23) -- (s123)
node[pos=0.38,markstyle] {\(\scriptstyle\vartriangleright\)}
node[pos=0.62,markstyle] {\(\scriptstyle\vartriangleright\vartriangleright\)};
\draw[edge] (s) -- (s2)
node[pos=0.38,markstyle] {\(\scriptstyle\blacktriangleright\)}
node[pos=0.62,markstyle] {\(\scriptstyle\vartriangleright\)};
\draw[edge] (s1) -- (s12)
node[pos=0.38,markstyle] {\(\scriptstyle\blacktriangleright\)}
node[pos=0.62,markstyle] {\(\scriptstyle\vartriangleright\)};
\draw[edge] (s3) -- (s23)
node[pos=0.38,markstyle] {\(\scriptstyle\vartriangleright\)}
node[pos=0.62,markstyle] {\(\scriptstyle\blacktriangleright\)};
\draw[edge] (s13) -- (s123)
node[pos=0.38,markstyle] {\(\scriptstyle\vartriangleright\)}
node[pos=0.62,markstyle] {\(\scriptstyle\blacktriangleright\)};
\draw[edge] (s) -- (s3)
node[pos=0.38,markstyle] {\(\scriptstyle\blacktriangleright\)}
node[pos=0.62,markstyle] {\(\scriptstyle\blacktriangleright\blacktriangleright\)};
\draw[edge] (s1) -- (s13)
node[pos=0.38,markstyle] {\(\scriptstyle\blacktriangleright\)}
node[pos=0.62,markstyle] {\(\scriptstyle\blacktriangleright\blacktriangleright\)};
\draw[edge] (s2) -- (s23)
node[pos=0.38,markstyle] {\(\scriptstyle\blacktriangleright\)}
node[pos=0.62,markstyle] {\(\scriptstyle\blacktriangleright\blacktriangleright\)};
\draw[edge] (s12) -- (s123)
node[pos=0.38,markstyle] {\(\scriptstyle\blacktriangleright\)}
node[pos=0.62,markstyle] {\(\scriptstyle\blacktriangleright\blacktriangleright\)};
\node[below=8pt] at (s)
{$\fs$};
\node[left=8pt] at (s1)
{$\fs_1$};
\node[below right=2pt] at (s2)
{$\fs_2$};
\node[right=8pt] at (s3)
{$\fs_3$};
\node[left=8pt] at (s12)
{$\fs_{12}$};
\node[above right=2pt] at (s13)
{$\fs_{13}$};
\node[right=8pt] at (s23)
{$\fs_{23}$};
\node[above=8pt] at (s123)
{$\fs_{123}$};
\filldraw[black] (s) circle (2pt);
\filldraw[black] (s1) circle (2pt);
\filldraw[black] (s2) circle (2pt);
\filldraw[black] (s3) circle (2pt);
\filldraw[black] (s12) circle (2pt);
\filldraw[black] (s13) circle (2pt);
\filldraw[black] (s23) circle (2pt);
\filldraw[black] (s123) circle (2pt);
\end{tikzpicture}
\[
\begin{aligned}
\fs      &= (\fh, {}^{q+}\fu, {}^{q+}\fv, {}^{q+}\fw), &
\fs_1    &= (\fh, q,\fu, {}^{q+}\fv, {}^{q+}\fw), \\
\fs_2    &= (\fh, {}^{q+}\fu, q,\fv, {}^{q+}\fw), &
\fs_3    &= (\fh, {}^{q+}\fu, {}^{q+}\fv, q,\fw), \\
\fs_{12} &= (\fh, q,\fu, q,\fv, {}^{q+}\fw), &
\fs_{13} &= (\fh, q,\fu, {}^{q+}\fv, q,\fw), \\
\fs_{23} &= (\fh, {}^{q+}\fu, q,\fv, q,\fw), &
\fs_{123}&= (\fh, q,\fu, q,\fv, q,\fw).
\end{aligned}
\]
\caption{
The cube underlying Lemma~\ref{lem_cube}. 
In the associated quadruple-carry relation, $\blacktriangleright$ denotes the positive sign and $\vartriangleright$ denotes the negative sign. 
The double arrows $\blacktriangleright\blacktriangleright$ and $\vartriangleright\vartriangleright$ represent the positive sign and negative sign of this quadruple‑carry relation within the cube identity.
}
\label{fig:cube-identity}
\end{figure}
\end{remark}

Recall Definition \ref{def_Qw}. 
We have
\begin{align*}
\HLw &= \left\{ \bQ\left( \fh, \boxed{{}^{q+}\fm},\boxed{{}^{q+}\ft} \right) ~\Big|~ \fh \in \cI, \fm \in \cI_{>0}, \ft \in \cI, \wt(\fh,{}^{q+}\fm,{}^{q+}\ft) = w \right\}, \\ 
\Qw &= \Span_{\bbF_q}\HLw.
\end{align*}

\begin{definition}\label{def_Qg}
For any $\fh \in [q]^{\bullet}$ and $\fm \in \cI_{>0}$, define the subspace $\mathscr{Q}_w^{\fh,\fm}$ of $\Qw$ by
\[
\mathscr{Q}_w^{\fh,\fm} \coloneq \Span_{\bbF_q}\left\{ \bQ\left( \fh, \boxed{{}^{q+}\fm},\boxed{{}^{q+}\ft} \right) ~\Big|~ \ft \in \cI, \wt(\fh,{}^{q+}\fm,{}^{q+}\ft) = w  \right\}.
\]
\end{definition}
The point of Definition \ref{def_Qg} is that we restrict $\fh \in [q]^{\bullet}$.

\begin{lemma}[Initial normalization of quadruple-carry relations]\label{lem_ini_nor}
For any $\fg \in \cI$, $\fa \in \cI_{>0}$, and $\fb \in \cI$ with $\wt(\fg,{}^{q+}\fa,{}^{q+}\fb)=w$, we have
\begin{equation}\label{eqn_carno_by_cube}
\bQ\left(\fg,\boxed{{}^{q+}\fa},\boxed{{}^{q+}\fb}\right) \in \sum_{\substack{\fh \in [q]^{\bullet}, \fm \in \cI_{>0} \\ (\fh,{}^{q+}\fm) \succcurlyeq \fg}} \mathscr{Q}_w^{\fh,\fm}. 
\end{equation}
In particular,
\begin{equation}\label{eqn_carnoQw}
\Qw = \sum_{\fh \in [q]^{\bullet}, \fm \in \cI_{>0}} \mathscr{Q}_w^{\fh,\fm}. 
\end{equation}
\end{lemma}

\begin{proof}
If $\fg \in [q]^{\bullet}$, then $\bQ\left(\fg,\boxed{{}^{q+}\fa},\boxed{{}^{q+}\fb}\right) \in \mathscr{Q}_w^{\fg,\fa}$ and \eqref{eqn_carno_by_cube} holds (since $(\fg,{}^{q+}\fa) \succ \fg$).
If $\fg \notin [q]^{\bullet}$, then we can write $\fg=(\fh_0,{}^{q+}\fm_0)$ for some $\fh_0 \in [q]^{\bullet}$ and $\fm_0 \in \cI_{>0}$.
By Lemma \ref{lem_cube}, we have
\[
\bQ\left(\fg,\boxed{{}^{q+}\fa},\boxed{{}^{q+}\fb}\right) \in -\bQ\left(\fg^\prime,\boxed{{}^{q+}\fa},\boxed{{}^{q+}\fb}\right) + \sum_{\substack{\fh \in [q]^{\bullet}, \fm \in \cI_{>0} \\ (\fh,{}^{q+}\fm) \succcurlyeq \fg}} \mathscr{Q}_w^{\fh,\fm},
\]
where $\fg^\prime = (\fh_0, q,\fm_0) \succ \fg$.
Repeat the process on $\bQ\left(\fg^\prime,\boxed{{}^{q+}\fa},\boxed{{}^{q+}\fb}\right)$. The process terminates in finitely many steps, thereby completing the proof of \eqref{eqn_carno_by_cube}.
The second assertion \eqref{eqn_carnoQw} follows from \eqref{eqn_carno_by_cube}.
\end{proof}

\subsection{The linear operators $\sfE$, $\sfC$, $\sfL$, and $\sfZ$}

Since we need to deal with some identities involving too many quadruple-carry relations in the next subsection, we introduce the following more compact notation.
\begin{definition}\label{def_ECL}
Define the linear operators $\sfE$, $\sfC$, $\sfL$: $\mathscr{S}^{(\bbF_q)} \rightarrow \mathscr{S}^{(\bbF_q)}$ by setting for each $\fs \in \cI$ that
\begin{align*}
\sfE\fs &\coloneq \mathbbm{1}_{\fs \neq \emptyset} \cdot \fs^{+} + (\fs,1), \\
\sfC\fs &\coloneq {}^{q+}\fs + \mathbbm{1}_{\fs \neq \emptyset} \cdot (q,\fs), \\
\sfL\fs &\coloneq (q-1)*\fs. 
\end{align*}
\end{definition}

\begin{example}[A compact form of quadruple-carry relations]\label{lem_compact_form_of_Q}
By the definition of quadruple-carry relations (see Lemma \ref{lem_HL}) and Definition \ref{def_ECL},
for any $\fh, \fv \in \cI$ and $\fu \in \cI_{>0}$, we have
\[
\bQ\left( \fh, \boxed{{}^{q+}\fu},\boxed{{}^{q+}\fv} \right) = \Big( \sfE(\fh,\sfC\fu), \sfL\fv \Big) - \Big( \sfE\fh,\sfL(\fu,\sfC\fv) \Big).
\]
\end{example}

\begin{example}[A compact form of the cube identity]
By Lemma \ref{lem_cube} and Definition \ref{def_ECL}, for any $\fh, \fw \in \cI$ and $\fu, \fv \in \cI_{>0}$, we have
\[
\bQ\left( \fh,\sfC\fu, \boxed{{}^{q+}\fv}, \boxed{{}^{q+}\fw} \right) -\bQ\left( \fh, \boxed{{}^{q+}\fu},\sfC\fv,\boxed{{}^{q+}\fw} \right) + \bQ\left( \fh, \boxed{{}^{q+}\fu}, \boxed{{}^{q+}\fv}, \sfC\fw \right) = \bzero.
\]
\end{example}

The following linear operator $\sfZ$ is important.

\begin{definition}\label{def_Z_operator}
Define the linear operator $\sfZ$: $\Sw \rightarrow \Sw$ by setting for each $\fs \in \cI_w$ that
\[
\sfZ\fs \coloneq \sum_{\ft \in \operatorname{Coars}(\fs)} \ft \in \Sw,
\]
where
\[
\operatorname{Coars}(\fs) \coloneq \left\{ (s_1 \circ_1 s_2 \circ_2 s_3 \cdots \circ_{k-1} s_k) ~\Big|~ \text{each $\circ_i$ is either the comma , or the plus +} \right\}
\]
for $\fs=(s_1,\ldots,s_k)$ with $\dep(\fs) = k \ge 2$,
with the convention $\operatorname{Coars}(\fs) = \{\fs\}$ for $\dep(\fs) \le 1$.
\end{definition}

The following lemma is well known.

\begin{lemma}\label{lem_Z}
The linear operator $\sfZ$ is invertible; its inverse linear operator $\sfZ^{-1}$ satisfies
\[
\sfZ^{-1}\fs = \sum_{\ft \in \operatorname{Coars}(\fs)} (-1)^{\dep(\fs)-\dep(\ft)} \ft
\]
for any $\fs \in \cI_w$.
Moreover, we have
\[
\LiA(\sfZ\fs) = \LiAS(\fs), \quad \LiAS(\sfZ^{-1}\fs)=\LiA(\fs).
\]
\end{lemma}

\subsection{The six-family-$\bQ$ identity}

The next lemma is a straightforward computation.
We state it in a compact form using the operators $\sfE$, $\sfC$, and $\sfL$ defined in Definition \ref{def_ECL}.
As mentioned in Definition \ref{def_q+}, whenever a tuple-valued argument of $\bQ, \sfE,\sfC,\sfL,\sfZ,{}^{q+}(\cdot),(\cdot)^{+}$ is a formal sum, the resulting expression is interpreted multilinearly.
For example, $\sfL(\fs,c_1\ft_1+c_2\ft_2) = c_1\sfL(\fs,\ft_1) + c_2\sfL(\fs,\ft_2)$ and $(c_1\ft_1+c_2\ft_2)^{+}=c_1\ft_1^{+}+c_2\ft_2^{+}$ for any $\fs,\ft_1,\ft_2 \in \cI$ and $c_1,c_2 \in \bbF_q$.
\begin{lemma}
For any formal sums $\bA, \bB \in \mathscr{S}^{(\bbF_q)}$, we have
\begin{align}
\sfL\left(\sfE\bA,\bB\right) &= \left( \sfE\bA, \sfL\bB \right) + \left( \sfE\sfL\bA, \bB \right), \label{eqn_ELC_1}\\
\sfL(\bA,\sfC\bB) &= \left( \sfL\bA, \sfC\bB \right) + \left( \bA, \sfC\sfL\bB \right), \label{eqn_ELC_2}\\
\sfE\sfC\bA &= \sfC\sfE\bA. \label{eqn_EC=CE}
\end{align}
\end{lemma}

\begin{proof}
By multilinearity, it suffices to prove the case $\bA = \fa \in \cI$ and $\bB=\fb \in \cI$.
Note that for any $\fa,\fb \in \cI$ and $m \in \bbZ_+$, by the definition of stuffle product, we have
\[
\sfL(\fa,m,\fb) = (\sfL\fa,m,\fb) + (\fa,m+q-1,\fb) + (\fa,m,\sfL\fb).
\]

Using the above equality, we first prove \eqref{eqn_ELC_1} for the case $\fa \neq \emptyset$.
Assume $\fa=(\fa^\prime,a_k)$, where $\fa^\prime\in\cI$ and $a_k \in \bbZ_{+}$. 
Then
\begin{align*}
\sfL(\sfE\fa,\fb) &= \sfL(\fa^+,\fb) + \sfL(\fa,1,\fb) \\
&=(\sfL\fa^\prime,a_k+1,\fb) + (\fa^\prime,a_k+q,\fb) + (\fa^{+},\sfL\fb) + (\sfL\fa,1,\fb) + (\fa,q,\fb) + (\fa,1,\sfL\fb).
\end{align*}
On the other hand,
\begin{align*}
(\sfE\fa,\sfL\fb)&=(\fa^+,\sfL\fb) + (\fa,1,\sfL\fb), \\
(\sfE\sfL\fa,\fb) &= (\sfE(\sfL\fa^\prime,a_k),\fb) + (\sfE(\fa^\prime,a_k+q-1),\fb) + (\sfE(\fa,q-1),\fb) \\
&= (\sfL\fa^\prime,a_k+1,\fb) + (\sfL\fa^\prime,a_k,1,\fb) +(\fa^\prime,a_k+q,\fb) + (\fa^\prime,a_k+q-1,1,\fb) \\
&\quad\ +(\fa,q,\fb) + (\fa,q-1,1,\fb).
\end{align*}
Noting that $\sfL\fa=(\sfL\fa^\prime,a_k)+(\fa^\prime,a_k+q-1)+(\fa,q-1)$, we deduce from the above equations that $\sfL(\sfE\fa,\fb) = (\sfE\fa,\sfL\fb) + (\sfE\sfL\fa,\fb)$.
For the case $\fa=\emptyset$, it is easier to check that $\sfL(\sfE\fa,\fb)$ and $(\sfE\fa,\sfL\fb) + (\sfE\sfL\fa,\fb)$ are both equal to $(q-1,1,\fb)+(q,\fb)+(1,\sfL\fb)$.
Therefore, the proof of \eqref{eqn_ELC_1} is complete. 

The proofs of \eqref{eqn_ELC_2} and \eqref{eqn_EC=CE} are similar to the above. 
We omit the lengthy computational details here.
\end{proof}

The next Lemma \ref{lem_6Q} (six-family-$\bQ$ identity) is contributed by ChatGPT (GPT-5.6 Sol, OpenAI).
It (together with the cube identity) provides enough identities among quadruple-carry relations to obtain an upper bound for $\dim_{\bbF_q} \Qw$.

Let $\mathscr{S}^{(\bbF_q)}_{>0} = \bigoplus_{w=1}^{\infty} \Sw$ be the $\bbF_q$-linear space of formal sums of non-empty tuples.
\begin{lemma}[The six-family-$\bQ$ identity]\label{lem_6Q}
For any formal sums $\bH,\bB \in \mathscr{S}^{(\bbF_q)}$ and $\bM,\bA \in \mathscr{S}_{>0}^{(\bbF_q)}$, we have
\begin{align}
&\bQ\left( \bH, \boxed{{}^{q+}\bM}, \boxed{{}^{q+}(\sfE\bA,\sfL\bB)} \right) - \bQ\left( \bH, \boxed{{}^{q+}(\sfE\bM,\bA)},\boxed{{}^{q+}\sfL\bB} \right) \notag\\
-&\bQ\left( \bH,\boxed{{}^{q+}(\sfE\bM,\sfL\bA)},\boxed{{}^{q+}\bB} \right) + \bQ\left( \sfE\bH,\bM,\boxed{{}^{q+}\bA},\boxed{{}^{q+}\sfL\bB} \right) \notag\\
+&\bQ\left( \sfE\bH,\bM,\boxed{{}^{q+}\sfL\bA},\boxed{{}^{q+}\bB} \right) +\bQ\left( \sfE\bH,\sfL\bM,\boxed{{}^{q+}\bA},\boxed{{}^{q+}\bB} \right) = \bzero. \label{eqn_6Q}
\end{align}
\end{lemma}

\begin{proof}
By Example \ref{lem_compact_form_of_Q}, for any $\bM \in \mathscr{S}_{>0}^{(\bbF_q)}$ and $\bH,\bT \in \mathscr{S}^{(\bbF_q)}$, we have
\begin{equation}\label{eqn_Q_to_ECL}
\bQ\left( \bH,\boxed{{}^{q+}\bM},\boxed{{}^{q+}\bT} \right) = \Big( \bH,\sfE\sfC\bM,\sfL\bT \Big) - \Big( \sfE\bH, \sfL\left(\bM,\sfC\bT \right) \Big).
\end{equation}
By \eqref{eqn_EC=CE}, we have $\sfE\sfC=\sfC\sfE$; denote this operator by $\sfD$.
Using \eqref{eqn_Q_to_ECL} to expand each term on the left-hand side of \eqref{eqn_6Q} and collecting terms by multilinearity, we obtain
\[
\text{LHS of \eqref{eqn_6Q}} = \left(\bH,\sfD\bM,\bS_1\right) + \left(\sfE\bH,\bS_2\right),
\]
where
\begin{align}
\bS_1 =&~ \sfL\left( \sfE\bA,\sfL\bB \right) - \left( \sfE\bA, \sfL\sfL\bB \right) - \left( \sfE\sfL\bA,\sfL\bB \right), \notag\\
\bS_2 =&~ -\sfL\Big( \sfE(\bM,\sfC\bA),\sfL\bB \Big) + \sfL\Big( \sfE\bM, \bA, \sfC\sfL\bB \Big) + \sfL\Big( \sfE\bM, \sfL\bA, \sfC\bB \Big) \notag\\
&~ +\Big( \sfE(\bM,\sfC\bA),\sfL\sfL\bB \Big) - \Big( \sfE\bM, \sfL(\bA,\sfC\sfL\bB) \Big) \notag\\
&~ +\Big( \sfE(\bM,\sfC\sfL\bA), \sfL\bB \Big) - \Big( \sfE\bM,\sfL(\sfL\bA,\sfC\bB) \Big) \notag\\
&~ +\Big( \sfE(\sfL\bM,\sfC\bA),\sfL\bB \Big) - \Big( \sfE\sfL\bM, \sfL(\bA,\sfC\bB) \Big). \label{eqn_S2}
\end{align}

By \eqref{eqn_ELC_1}, we have $\bS_1=\bzero$.
By \eqref{eqn_ELC_1} and \eqref{eqn_ELC_2}, we have
\begin{align}
-\sfL\Big( \sfE(\bM,\sfC\bA),\sfL\bB \Big)   \overset{\eqref{eqn_ELC_1}}{=}&~ -\Big( \sfE(\bM,\sfC\bA),\sfL\sfL\bB \Big) - \Big( \sfE\sfL(\bM,\sfC\bA),\sfL\bB \Big)  \notag\\
\overset{\eqref{eqn_ELC_2}}{=}&~ -\Big( \sfE(\bM,\sfC\bA),\sfL\sfL\bB \Big) - \Big( \sfE(\sfL\bM,\sfC\bA),\sfL\bB \Big) - \Big( \sfE(\bM,\sfC\sfL\bA),\sfL\bB \Big), \label{eqn_L1} \\
\sfL\Big( \sfE\bM, \bA, \sfC\sfL\bB \Big)  \overset{\eqref{eqn_ELC_1}}{=}&~ \Big( \sfE\bM, \sfL(\bA,\sfC\sfL\bB) \Big) + \Big( \sfE\sfL\bM, \bA, \sfC\sfL\bB \Big), \label{eqn_L2} \\
\sfL\Big( \sfE\bM, \sfL\bA, \sfC\bB \Big) \overset{\eqref{eqn_ELC_1}}{=}&~ \Big( \sfE\bM,\sfL(\sfL\bA,\sfC\bB) \Big) + \Big( \sfE\sfL\bM, \sfL\bA, \sfC\bB \Big). \label{eqn_L3}
\end{align}
Substituting \eqref{eqn_L1}--\eqref{eqn_L3} into \eqref{eqn_S2}, we obtain
\[
\bS_2 = - \Big( \sfE\sfL\bM, \sfL(\bA,\sfC\bB) \Big) + \Big( \sfE\sfL\bM, \bA, \sfC\sfL\bB \Big) + \Big( \sfE\sfL\bM, \sfL\bA, \sfC\bB \Big).
\]
Applying \eqref{eqn_ELC_2} again to $\sfL(\bA,\sfC\bB)$, we deduce that $\bS_2=\bzero$.
The proof of \eqref{eqn_6Q} is complete.
\end{proof}

\subsection{Another basis of $\mathscr{S}^{(\bbF_q)}$}
Note that by definition, $\cI$ is a basis of $\mathscr{S}^{(\bbF_q)}$.
In this subsection, we make use of another basis of $\mathscr{S}^{(\bbF_q)}$, together with the six-family-$\bQ$ identity (Lemma \ref{lem_6Q}) and the initial normalization lemma (Lemma \ref{lem_ini_nor}) to prove that $\dim_{\bbF_q} \Qw \leqslant 2^{w-1} - d_w$ for any weight $w \ge 1$.

Note that for any tuple $\fs \in \cI_{>0}$, we have 
\[
\sfL\fs = (q-1)*\fs = (q-1)\shuffle\fs + (q-1)\widetilde{*}\fs, 
\]
where $\shuffle$ is the shuffle product and $\widetilde{*}$ is the sticky product, see Remark \ref{rmk_HL}.
Now, we define a related linear operator $\sfLd: \mathscr{S}^{(\bbF_q)} \rightarrow \mathscr{S}^{(\bbF_q)}$ by setting for each $\fs \in \cI_{>0}$ that
\[
\sfLd\fs \coloneq (q-1)\shuffle\fs - (q-1)\widetilde{*}\fs,
\]
and by setting $\sfLd\emptyset \coloneq (q-1)$.

For any non-empty tuples $\fa=(a_1,\ldots,a_k),\fb=(b_1,\ldots,b_\ell) \in \cI_{>0}$, define 
\[
\fa \boxplus \fb \coloneq (a_1,\ldots,a_{k-1},a_k+b_1,b_2,\ldots,b_\ell).
\] 
By convention, $\emptyset\boxplus\fa=\fa\boxplus\emptyset\coloneq\bzero \in \mathscr{S}^{(\bbF_q)}$ for any $\fa \in \cI$.

\begin{lemma}\label{lem_Ld}
For any $\bS \in \mathscr{S}^{(\bbF_q)}$, we have
\begin{equation}\label{eqn_LZ=ZLd}
\sfL\sfZ\bS = \sfZ\sfLd\bS.
\end{equation}
\end{lemma}

\begin{proof}
In a coarsening of $(\fa,\fb)$, the comma between $\fa$ and $\fb$ is either retained or replaced by the plus symbol.
Therefore, for any $\fa,\fb \in \cI$, we have
\begin{equation}\label{eqn_Z(a,b)}
\sfZ(\fa,\fb) = (\sfZ\fa,\sfZ\fb) + \sfZ(\fa\boxplus\fb).
\end{equation}

For any $a \in \bbZ_{+}$, define a linear operator $\sfE_a$: $\mathscr{S}^{(\bbF_q)} \rightarrow \mathscr{S}^{(\bbF_q)}$ by setting for each $\fs \in \cI$ that
\[
\sfE_a\fs \coloneq \mathbbm{1}_{\fs \neq \emptyset} \cdot \fs^{+a} + (\fs,a).
\]
Note that $\sfE_1 = \sfE$.
We claim that, for any formal sum $\bS \in \mathscr{S}^{(\bbF_q)}$, we have
\begin{align}
\sfZ(\bS,a) &= \sfE_a\sfZ\bS,  \label{eqn_Ea1} \\
\sfL\sfE_a\bS &= (\sfE_a\bS,q-1) + \sfE_a\sfL\bS. \label{eqn_Ea2}
\end{align}
It suffices to verify \eqref{eqn_Ea1} and \eqref{eqn_Ea2} for the case $\bS=\fs \in \cI$.
If $\fs=\emptyset$, then the verifications are straightforward. 
Now, we prove \eqref{eqn_Ea1} for the case $\bS = \fs \in \cI_{>0}$.
By \eqref{eqn_Z(a,b)}, we have
\[
\sfZ(\fs,a) = \Big(\sfZ\fs, a\Big) + \sfZ(\fs^{+a}).
\]
By the definition of $\operatorname{Coars}(\fs)$ (see Definition \ref{def_Z_operator}), we have $\left\{ \ft^{+a} \mid \ft \in \operatorname{Coars}(\fs) \right\} = \operatorname{Coars}(\fs^{+a})$.
Therefore, 
\[
\sfZ(\fs^{+a}) = (\sfZ\fs)^{+a}.
\]
On the other hand, by definition of $\sfE_a$, we have (since $\fs \in \cI_{>0}$ and $\sfZ\fs \in \mathscr{S}_{>0}^{(\bbF_q)}$)
\[
\sfE_a\sfZ\fs = (\sfZ\fs)^{+a} + (\sfZ\fs,a).
\]
In conclusion, the proof of \eqref{eqn_Ea1} is complete.
It remains to prove \eqref{eqn_Ea2} for the case $\bS=\fs \in \cI_{>0}$.
Assume $\fs=(\ft,b)$ where $\ft \in \cI$ and $b \in \bbZ_+$.
Then
\begin{align}
\sfL\sfE_a(\ft,b) =&~ \sfL\Big( (\ft,b+a) + (\ft,b,a) \Big) \notag\\
=&~ (\sfL\ft,b+a) + (\ft,b+a+q-1) + (\ft,b+a,q-1)  \notag\\
&+~ (\sfL\ft,b,a) +(\ft,b+q-1,a) + (\ft,b,q-1,a) + (\ft,b,a+q-1) + (\ft,b,a,q-1). \label{eqn_LEa}
\end{align}
On the other hand, 
\begin{align}
(\sfE_a(\ft,b),q-1) =&~ (\ft,b+a,q-1) + (\ft,b,a,q-1), \label{eqn_LEa1}\\
\sfE_a\sfL(\ft,b) =&~ \sfE_a\Big( (\sfL\ft,b) +(\ft,b+q-1) + (\ft,b,q-1) \Big) \notag\\
=&~ (\sfL\ft,b+a) + (\sfL\ft,b,a) + (\ft,b+a+q-1) + (\ft,b+q-1,a)  \notag\\ 
&~ +(\ft,b,a+q-1) +(\ft,b,q-1,a). \label{eqn_LEa2}
\end{align}
Comparing the sum of \eqref{eqn_LEa1} and \eqref{eqn_LEa2} with \eqref{eqn_LEa}, we complete the proof of \eqref{eqn_Ea2}.

Finally, we prove \eqref{eqn_LZ=ZLd}.
It suffices to prove that $\sfL\sfZ\fs = \sfZ\sfLd\fs$ for any $\fs \in \cI$.
We prove by induction on $\dep(\fs)$.
For $\dep(\fs) \le 1$, this is readily checkable.
Now, suppose $\dep(\fs) \ge 2$ and write $\fs = (\ft,a)$ for some $a \in \bbZ_{+}$ and $\ft \in \cI_{>0}$.
By definition of $\sfLd$ and \eqref{eqn_Z(a,b)}, we have
\begin{equation}\label{eqn_8.1}
\sfLd(\ft,a) = (\sfLd\ft,a) - (\ft,a+q-1) + (\ft,a,q-1).
\end{equation}
Applying the operator $\sfZ$ to \eqref{eqn_8.1}, and noting that $\sfZ(\sfLd\ft,a) = \sfE_a\sfZ\sfLd\ft$ by \eqref{eqn_Ea1} and 
\[
\sfZ \Big(- (\ft,a+q-1) + (\ft,a,q-1)\Big) = (\sfZ(\ft,a),q-1) = (\sfE_a\sfZ\ft,q-1)
\]
first by \eqref{eqn_Z(a,b)} and then by \eqref{eqn_Ea1}, we obtain
\begin{equation}\label{eqn_8.2}
\sfZ\sfLd(\ft,a) = \sfE_a\sfZ\sfLd\ft + (\sfE_a\sfZ\ft,q-1).
\end{equation}
On the other hand, we have
\begin{equation}\label{eqn_8.3}
\sfL\sfZ(\ft,a) = \sfL\sfE_a\sfZ\ft = (\sfE_a\sfZ\ft,q-1) + \sfE_a\sfL\sfZ\ft
\end{equation}
first by \eqref{eqn_Ea1} and then by \eqref{eqn_Ea2}.
Now, comparing \eqref{eqn_8.2} with \eqref{eqn_8.3}, and using the induction hypothesis $\sfL\sfZ\ft = \sfZ\sfLd\ft$, we complete the proof of \eqref{eqn_LZ=ZLd}.
\end{proof}

We are ready to construct a new basis of $\mathscr{S}^{(\bbF_q)}$.

\begin{definition}\label{def_N}
Define
\[
\cN \coloneq \{\emptyset\} \bigcup \bbZ_+ \bigcup \left\{ (h,\fm,t) ~\Big|~ h \in \bbZ_+, \fm \in [q]^{\bullet}, t \in [q-1]  \right\} \subset \cI,
\]
and define
\[
\cJ_w \coloneq \left\{ \sfZ\fb ~\Big|~ \fb \in \cN, \wt(\fb)=w  \right\} \bigcup \left\{ (\sfE\sfZ\fa,\sfL\sfZ\fb) ~\Big|~ \fa \in \cI_{>0},\fb \in \cN, \wt(\fa)+\wt(\fb)+q=w \right\}.
\]
\end{definition}

\begin{lemma}\label{lem_J}
$\cJ_w$ is a basis of the $\bbF_q$-linear space $\Sw$.
\end{lemma}

\begin{proof}
We first claim that 
\begin{equation}\label{eqn_8.10}
\sfZ^{-1}(\sfE\sfZ\fa,\fs) = (\fa,1,\sfZ^{-1}\fs)  -(\fa,{}^{+}(\sfZ^{-1}\fs))
\end{equation}
for any $\fa,\fs \in \cI_{>0}$.
Indeed, by \eqref{eqn_Z(a,b)}, we have
\begin{equation}\label{eqn_8.11}
\sfZ\Big( (\fa,1,\sfZ^{-1}\fs)  -(\fa,{}^{+}(\sfZ^{-1}\fs)) \Big) = \left( \sfZ(\fa,1),\fs \right).
\end{equation}
By the special case $a=1$ of \eqref{eqn_Ea1}, we have $\sfZ(\fa,1)=\sfE\sfZ\fa$. 
Substituting it into \eqref{eqn_8.11}, we obtain
\begin{equation}\label{eqn_8.12}
\sfZ\Big( (\fa,1,\sfZ^{-1}\fs)  -(\fa,{}^{+}(\sfZ^{-1}\fs)) \Big) = \left( \sfE\sfZ\fa,\fs \right).
\end{equation}
Applying the operator $\sfZ^{-1}$ to \eqref{eqn_8.12}, we obtain the desired \eqref{eqn_8.10}.

Therefore, by \eqref{eqn_8.10} and Lemma \ref{lem_Ld}, we have
\begin{align*}
\sfZ^{-1}\cJ_w &= \left\{ \fb ~\Big|~ \fb \in \cN\cap\cI_w \right\} \bigcup \left\{ \sfZ^{-1}(\sfE\sfZ\fa,\sfL\sfZ\fb) ~\Big|~ \fa \in \cI_{>0},\fb \in \cN,\wt(\fa)+\wt(\fb)+q=w \right\} \\
&= \left\{ \fb ~\Big|~ \fb \in \cN\cap\cI_w  \right\} \bigcup \left\{ (\fa,1,\sfLd\fb) - (\fa,{}^{+}(\sfLd\fb)) ~\Big|~ \fa \in \cI_{>0},\fb \in \cN, \wt(\fa)+\wt(\fb)+q=w \right\}.
\end{align*}
Define the (\text{depth}, $\overleftarrow{\text{lex}})$-order $\triangleleft$ on $\cI$: For any $\fs,\ft \in \cI$, say $\fs \triangleleft \ft$ if and only if
\begin{itemize}
\item $\dep(\fs)<\dep(\ft)$; or
\item $\dep(\fs)=\dep(\ft)$ and $\fs$ is smaller than $\ft$ in the reverse lexicographic order (compare entries from right to left).
\end{itemize}
For any $\bS \in \Sw \setminus \{\bzero\}$, define $s_{\triangleleft}(\bS)$ to be the minimal tuple appearing in $\bS$ with respect to the (\text{depth},  $\overleftarrow{\text{lex}})$-order $\triangleleft$.
By the definition of $\sfLd$, for any $\fa \in \cI_{>0}$ and $\fb \in \cN$, we have
\[
s_{\triangleleft}\Big( (\fa,1,\sfLd\fb) - (\fa,{}^{+}(\sfLd\fb)) \Big) = (\fa,{}^{q+}\fb).
\]
Clearly, for different ordered pairs of $\fa \in \cI_{>0}$ and $\fb \in \cN$, the resulting tuples $(\fa,{}^{q+}\fb)$ are distinct (by comparing the components from right to left).
Note that 
\begin{align*}
&~ \left\{ (\fa,{}^{q+}\fb) ~\Big|~ \fa\in \cI_{>0}, \fb \in \cN \right\} \\
=&~ \left\{ (\fa,t) ~\Big|~ \fa \in \cI_{>0}, t \ge q \right\} \bigcup \left\{ (\fa,h,\fm,t) ~\Big|~ \fa \in \cI_{>0}, h \ge q+1, \fm \in [q]^{\bullet}, t \in [q-1] \right\}  \\
=&~ \cI \setminus \cN.
\end{align*}
It implies that if $\bS$ runs through $\sfZ^{-1}\cJ_w$, then $s_{\triangleleft}(\bS)$ runs through $\cI_w$.
Therefore, the transform matrix from $\cI_w$ to $\sfZ^{-1}\cJ_w$ is triangular with diagonal entries in $\bbF_q^{\times}$ (actually in $\{\pm 1\}$).
Thus, $\sfZ^{-1}\cJ_w$ is a basis of $\Sw$ and hence $\cJ_w$ is a basis, too. 
\end{proof}

Now, we prove that $\dim_{\bbF_q} \Qw \leqslant 2^{w-1} - d_w$ for any weight $w \ge 1$.

\begin{proposition}\label{prop_other_half}
We have 
\begin{equation}\label{eqn_Qprime}
\Qw = \Span_{\bbF_q}\left\{  \bQ\left(\fh,\boxed{{}^{q+}\fm},\boxed{{}^{q+}\sfZ\fb}\right) ~\Big|~ \fh \in [q]^{\bullet}, \fm \in \cI_{>0}, \fb \in \cN, \wt(\fh,{}^{q+}\fm,{}^{q+}\fb)=w \right\}.
\end{equation}
In particular, for any weight $w \ge 1$, we have
\[
\dim_{\bbF_q} \Qw \leqslant 2^{w-1} - d_w.
\]
\end{proposition}

\begin{proof}
Denote the space on the right-hand side of \eqref{eqn_Qprime} by $\mathscr{Q}_w^{\prime}$. 
We first prove \eqref{eqn_Qprime}.
By \eqref{eqn_carnoQw}, it suffices to prove that
\[
\mathscr{Q}_w^{\fh,\fm} \subset \mathscr{Q}_w^{\prime}
\]
for any $\fh \in [q]^{\bullet}$ and $\fm \in \cI_{>0}$.
Recall the space $\mathscr{Q}_w^{\fh,\fm}$ is defined in Definition \ref{def_Qg}.
If $\wt(\fh) + \wt(\fm) > w-2q$, then $\mathscr{Q}_w^{\fh,\fm} = \Span_{\bbF_q} \emptyset = \{\bzero\}$; there is nothing to prove.
We may assume $w > 2q$.
In the following, we prove that $\mathscr{Q}_w^{\fh,\fm} \subset \mathscr{Q}_w^{\prime}$ by a reverse induction on $(\fh,{}^{q+}\fm) \in \mathcal{G}$ (with respect to the (depth, lex)-order $\prec$), where $\mathcal{G}$ is the finite set
\[
\mathcal{G} = \left\{ (\fh,{}^{q+}\fm) ~\Big|~  \fh \in [q]^{\bullet}, \fm \in \cI_{>0}, \wt(\fh) +\wt(\fm) \le w-2q \right\} \subset \cI.
\]

If $(\fh,{}^{q+}\fm) = (q+1,\{1\}^{w-2q-1})$ is the greatest tuple in $\mathcal{G}$, then $\fh = \emptyset$, $\fm=\{1\}^{w-2q}$, and
\[
\mathscr{Q}_w^{\fh,\fm} = \Span_{\bbF_q} \left\{ \bQ\left(\fh,\boxed{{}^{q+}\fm},\boxed{{}^{q+}\emptyset}\right)   \right\} \subset \mathscr{Q}_w^{\prime}.
\]
Now, assume we have proved that $\mathscr{Q}_w^{\fh^\prime,\fm^\prime} \subset \mathscr{Q}_w^{\prime}$
for any $\fh^\prime \in [q]^{\bullet}$ and $\fm^{\prime} \in \cI_{>0}$ such that $(\fh^\prime,{}^{q+}\fm^\prime) \succ (\fh,{}^{q+}\fm)$.
To prove that $\mathscr{Q}_w^{\fh,\fm} \subset \mathscr{Q}_w^{\prime}$, by the definition of $\mathscr{Q}_w^{\fh,\fm}$, it suffices to prove that
\[
\bQ\left( \fh, \boxed{{}^{q+}\fm},\boxed{{}^{q+}\ft} \right) \in \mathscr{Q}_w^{\prime}
\]
for any $\ft \in \cI$ such that $\wt(\fh,{}^{q+}\fm,{}^{q+}\ft)=w$.

By Lemma \ref{lem_J}, there exist $c_{\fb}, c_{\fa,\fb} \in \bbF_q$ such that
\[
\ft = \sum_{\fb \in \cN, \wt(\fb)=\wt(\ft)} c_{\fb}\cdot \sfZ\fb + \sum_{\substack{\fa\in\cI_{>0},\fb\in\cN \\ \wt(\fa)+\wt(\fb)+q=\wt(\ft)}} c_{\fa,\fb} \cdot (\sfE\sfZ\fa,\sfL\sfZ\fb).
\]
By definition, we have 
\[
\bQ\left( \fh, \boxed{{}^{q+}\fm},\boxed{{}^{q+}\sfZ\fb} \right) \in \mathscr{Q}_w^{\prime}.
\]
It remains to prove that
\[
\bQ\left( \fh, \boxed{{}^{q+}\fm},\boxed{{}^{q+}(\sfE\sfZ\fa,\sfL\sfZ\fb)} \right) \in \mathscr{Q}_w^{\prime}
\]
for any $\fa\in\cI_{>0}$, $\fb\in\cN$ with $\wt(\fa)+\wt(\fb)+q=\wt(\ft)$.

By Lemma \ref{lem_6Q}, we have
\begin{align}
&\bQ\left( \fh, \boxed{{}^{q+}\fm}, \boxed{{}^{q+}(\sfE\sfZ\fa,\sfL\sfZ\fb)} \right) = \bQ\left( \fh, \boxed{{}^{q+}(\sfE\fm,\sfZ\fa)},\boxed{{}^{q+}\sfL\sfZ\fb} \right) +\bQ\left( \fh,\boxed{{}^{q+}(\sfE\fm,\sfL\sfZ\fa)},\boxed{{}^{q+}\sfZ\fb} \right) \notag\\
&- \bQ\left( \sfE\fh,\fm,\boxed{{}^{q+}\sfZ\fa},\boxed{{}^{q+}\sfL\sfZ\fb} \right) -\bQ\left( \sfE\fh,\fm,\boxed{{}^{q+}\sfL\sfZ\fa},\boxed{{}^{q+}\sfZ\fb} \right) -\bQ\left( \sfE\fh,\sfL\fm,\boxed{{}^{q+}\sfZ\fa},\boxed{{}^{q+}\sfZ\fb} \right). \label{eqn_last6Q}
\end{align}

Since $\fh \in [q]^{\bullet}$ and every tuple appearing in $(\fh,{}^{q+}(\sfE\fm,\sfZ\fa))$ and $(\fh,{}^{q+}(\sfE\fm,\sfL\sfZ\fa))$ is $\succ (\fh,{}^{q+}\fm)$, by Definition \ref{def_Qg},
\begin{align*}
\bQ\left( \fh, \boxed{{}^{q+}(\sfE\fm,\sfZ\fa)},\boxed{{}^{q+}\sfL\sfZ\fb} \right) &\in \sum_{\fm^\prime \in \cI_{>0}, \fm^\prime \succ \fm} \mathscr{Q}_w^{\fh,\fm^\prime}, \\
\bQ\left( \fh,\boxed{{}^{q+}(\sfE\fm,\sfL\sfZ\fa)},\boxed{{}^{q+}\sfZ\fb} \right) &\in \sum_{\fm^\prime \in \cI_{>0}, \fm^\prime \succ \fm} \mathscr{Q}_w^{\fh,\fm^\prime}.
\end{align*}
Therefore, by induction hypothesis, the first two terms on the right-hand side of \eqref{eqn_last6Q} both belong to $\mathscr{Q}_w^{\prime}$. 

Since every tuple appearing in $(\sfE\fh,\fm)$ and $(\sfE\fh,\sfL\fm)$ is $\succ (\fh,{}^{q+}\fm)$, by \eqref{eqn_carno_by_cube} of Lemma \ref{lem_ini_nor}, the last three terms on the right-hand side of \eqref{eqn_last6Q} all belong to
\[
\sum_{\substack{\fh^{\prime} \in [q]^{\bullet}, \fm^\prime \in \cI_{>0} \\ (\fh^\prime,{}^{q+}\fm^\prime) \succ (\fh,{}^{q+}\fm)}} \mathscr{Q}_w^{\fh^\prime,\fm^\prime},
\]
and therefore belong to $\mathscr{Q}_w^{\prime}$ by induction hypothesis.

In conclusion, every term on the right-hand side of \eqref{eqn_last6Q} belongs to $\mathscr{Q}_w^{\prime}$.
Therefore, by \eqref{eqn_last6Q}, we have
\[
\bQ\left( \fh, \boxed{{}^{q+}\fm},\boxed{{}^{q+}(\sfE\sfZ\fa,\sfL\sfZ\fb)} \right) \in \mathscr{Q}_w^{\prime}
\]
as desired, which completes the reverse induction on $(\fh,{}^{q+}\fm)$. 
The proof of \eqref{eqn_Qprime} is now complete.

Finally, we prove the upper bound on $\dim_{\bbF_q} \Qw$. 
By \eqref{eqn_Qprime}, we have
\[
\dim_{\bbF_q} \Qw \le N_w \coloneq \#\left\{  \bQ\left(\fh,\boxed{{}^{q+}\fm},\boxed{{}^{q+}\sfZ\fb}\right) ~\Big|~ \fh \in [q]^{\bullet}, \fm \in \cI_{>0}, \fb \in \cN, \wt(\fh,{}^{q+}\fm,{}^{q+}\fb)=w \right\}.
\]
Clearly,
\[
N_w \le [x^{w}]\Big( \sum_{\fh \in [q]^{\bullet}} x^{\wt(\fh)} \cdot \sum_{\fm \in \cI_{>0}} x^{q+\wt(\fm)} \cdot \sum_{\fb \in \cN} x^{q+\wt(\fb)} \Big).
\]
Since 
\begin{align*}
\sum_{\fh \in [q]^{\bullet}} x^{\wt(\fh)} &= \frac{1}{1-x-x^2-\cdots-x^{q}} = \frac{1-x}{1-2x+x^{q+1}}, \\
\sum_{\fm \in \cI_{>0}} x^{\wt(\fm)} &= \frac{x}{1-2x}, \\
\sum_{\fb \in \cN} x^{\wt(\fb)} &= \frac{1}{1-x} + \frac{x}{1-x}\cdot\frac{1}{1-x-x^2-\cdots-x^{q}}\cdot\frac{x-x^q}{1-x} = \frac{1-x}{1-2x+x^{q+1}},
\end{align*}
we obtain
\begin{align*}
\dim_{\bbF_q} \Qw &\le [x^{w}] \frac{x^{2q+1}(1-x)^2}{(1-2x)(1-2x+x^{q+1})^2} \\
&= [x^w]\left( \frac{x}{1-2x} - \frac{x(1-x^{q})(1-2x+x^q)}{(1-2x+x^{q+1})^2} \right) =2^{w-1} - d_w.
\end{align*}
The proof of Proposition \ref{prop_other_half} is complete.
\end{proof}

\section{The proof of Theorem \ref{thm_Z}}\label{sec_9}

In this final section, we prove Theorem \ref{thm_Z}.
We separate the proof into two parts. 
The first part is achieved by the authors.
The second part (an elegant basis) is contributed by ChatGPT (GPT-5.6 Sol, OpenAI). 

\subsection{The dimension formula of $\cZ_{w}^{(\bbF_q)}$ and the first basis}

\begin{theorem}\label{thm_Z1}
For any weight $w \ge 1$, we have 
\[
\dim_{\bbF_q} \cZ_w^{(\bbF_q)} = d_w.
\]
Moreover, $\left\{ \LiA(\fs) ~\big|~ \fs \in \cB_w^{\prime} \right\}$ is an $\bbF_q$-basis of $\cZ_w^{(\bbF_q)}$, where $\cB_w^{\prime} = \cI_w \setminus \cM$, and $\cM$ is explicitly defined in Definition \ref{def_M}.
\end{theorem}

\begin{proof}
Consider the linear map 
\begin{align*}
\LiA \colon\ \Sw &\rightarrow K_\infty \\
\sum_{\fs \in \cI_w} c_{\fs} \cdot \fs &\mapsto  \sum_{\fs \in \cI_w} c_{\fs} \LiA(\fs).
\end{align*}
By Theorem \ref{thm_MZV=CMPLV}, its image is $\cZ_{w}^{(\bbF_q)}$.
Its kernel is $\Rw$ by definition.
Note that we have $\Rw =\Qw$ by Theorem \ref{thm_R2}.
Therefore, $\LiA$ induces an isomorphism $\overline{\LiA}$ between two $\bbF_q$-linear spaces:
\begin{align*}
\overline{\LiA} \colon\ \Sw/\Qw &\overset{\sim}{\rightarrow}  \cZ_{w}^{(\bbF_q)} \\
\sum_{\fs \in \cI_w} c_{\fs} \cdot \fs + \Qw &\mapsto  \sum_{\fs \in \cI_w} c_{\fs} \LiA(\fs).
\end{align*}
By Propositions \ref{prop_half} and \ref{prop_other_half}, we have $\dim_{\bbF_q}\Qw = 2^{w-1}-d_w$.
Since $\dim_{\bbF_q} \Sw = \#\cI_w = 2^{w-1}$, we obtain
\begin{equation}\label{eqn_9.1}
\dim_{\bbF_q} \cZ_w^{(\bbF_q)} = \dim_{\bbF_q} \Sw/\Qw = d_w.
\end{equation}

Moreover, for any $\fs \in \cM_w$, we have explicitly constructed an element $\bQ_{\fs} \in \Qw$ such that $s(\bQ_\fs)=\fs$ in Proposition \ref{prop_half}.
For any coset $\sum_{\fs \in \cI_w} c_{\fs} \cdot \fs + \Qw$ in the quotient space $\Sw/\Qw$, we can repeatedly add some $\bbF_q$-multiple of $\bQ_\fs$ to make the coefficient $c_{\fs}$ in the representative $\sum_{\fs \in \cI_w} c_{\fs} \cdot \fs$ become zero, from the minimal $\fs$ in $\cM_w$ to the maximal $\fs$ in $\cM_w$ (with respect to the (depth, lex)-order). 
Thus,
\[
\Sw/\Qw \subset \Span_{\bbF_q} \left\{ \fs + \Qw  ~\Big|~ \fs \in \cB_w^{\prime} \right\}.
\]
By Lemma \ref{lem_cardM} and \eqref{eqn_9.1}, the above inclusion must be an equality, and therefore 
\[
\left\{ \fs + \Qw  \mid \fs \in \cB_w^{\prime} \right\}
\]
is a basis of $\Sw/\Qw$.
Applying the isomorphism $\overline{\LiA}$, we conclude that 
\[
\left\{ \LiA(\fs)  \mid \fs \in \cB_w^{\prime} \right\}
\]
is a basis of $\cZ_{w}^{(\bbF_q)}$.
\end{proof}

\begin{remark}
Note that $\cM$ is explicitly given in Definition \ref{def_M}. 
Therefore, $\cB_w^{\prime} =\cI_w \setminus \cM$ is explicit.
In fact, we can write $\cB_w^{\prime}$ as a disjoint union of seven sets, such that each set can be described in one line.
We do not need to present the explicit description of $\cB_w^{\prime}$ here for the following reason.
After we found the basis $\left\{ \LiA(\fs)  \mid \fs \in \cB_w^{\prime} \right\}$ of $\cZ_{w}^{(\bbF_q)}$, we asked ChatGPT if there exists a more elegant basis.
ChatGPT (GPT-5.6 Sol, OpenAI) returned a positive answer; that is, the basis given in Theorem \ref{thm_Z}. 
We present it in the next subsection.
\end{remark}

\subsection{An elegant basis}

The following theorem is proved by ChatGPT (GPT-5.6 Sol, OpenAI). 
It has been digested and rewritten by the authors.

\begin{theorem}\label{thm_Z2}
$\left\{ \LiA^{\star}(\fs) ~\big|~ \fs \in \cB_w  \right\}$ is an $\bbF_q$-basis of $\cZ_w^{(\bbF_q)}$, where $\cB_w = \cB \cap \cI_w$ and
\[
\cB = [q]^{\bullet} \bigsqcup \left\{ (\fh,m,\ft) \in \cI ~\Big|~ \fh,\ft \in [q]^{\bullet}, m \in [2q] \setminus [q] \right\}.
\]
\end{theorem}

We need a last lemma. 
Recall that, for any $\bS \in \mathscr{S}^{(\bbF_q)} \setminus \{\bzero\}$, we denote by $s_{\triangleleft}(\bS)$ the minimal tuple appearing in $\bS$ with respect to the (\text{depth}, $\overleftarrow{\text{lex}})$-order $\triangleleft$.

\begin{lemma}\label{lem_last}
For any $\fh, \ft \in \cI$ and $\fm \in \cI_{>0}$, we have
\[
s_{\triangleleft}\Big( \sfZ^{-1}\bQ\left( \sfZ\fh, \boxed{{}^{q+}\sfZ\fm}, \boxed{{}^{q+}\sfZ\ft} \right) \Big) = \left( \fh, {}^{q+}\fm^{+q}, \ft \right).
\]
\end{lemma}

\begin{proof}
We first claim the following identities: For any $\bA,\bB \in \mathscr{S}^{(\bbF_q)}$, we have
\begin{align}
(\sfZ\bA,\sfZ\bB) &= \sfZ\Big( \left(\bA,\bB\right) - \bA \boxplus \bB  \Big), \label{921} \\
\sfE\sfZ\bA &=\sfZ(\bA,1), \label{922} \\
\sfC\sfZ\bA &= \sfZ(q,\bA). \label{923}
\end{align}
In fact, Equation \eqref{921} follows from \eqref{eqn_Z(a,b)}.
Equation \eqref{922} follows from the special case $a=1$ of \eqref{eqn_Ea1}.
Recall that by our definition, $q \boxplus \emptyset = \bzero$ and ${}^{q+} \emptyset = (q)$, so $\sfC\fs = (q,\fs) + q\boxplus\fs$ for any $\fs \in \cI$, including $\fs = \emptyset$. 
By \eqref{eqn_Z(a,b)}, we have
\[
\sfZ(q,\bA) = (q,\sfZ\bA) + \sfZ(q\boxplus\bA).
\]
For $\fa = \emptyset$ or $\dep(\fa)=1$, one checks directly $\sfZ(q\boxplus\fa)=q\boxplus\sfZ\fa$.
For $\dep(\fa) \ge 2$, by Definition \ref{def_Z_operator}, we have $\left\{ q \boxplus \ft \mid \ft \in \operatorname{Coars}(\fa) \right\} = \operatorname{Coars}(q \boxplus \fa)$. 
Therefore, $\sfZ(q\boxplus\bA) = q \boxplus \sfZ\bA$. 
Thus,
\[
\sfZ(q,\bA) = (q,\sfZ\bA) + q \boxplus \sfZ\bA = \sfC\sfZ\bA,
\]
which proves the last claim \eqref{923}.

Now, by Example \ref{lem_compact_form_of_Q}, we have
\begin{equation}\label{924}
\bQ\left( \sfZ\fh, \boxed{{}^{q+}\sfZ\fm},\boxed{{}^{q+}\sfZ\ft} \right) = \Big( \sfE(\sfZ\fh,\sfC\sfZ\fm), \sfL\sfZ\ft \Big) - \Big( \sfE\sfZ\fh,\sfL(\sfZ\fm,\sfC\sfZ\ft) \Big).
\end{equation}
Using the identities \eqref{921}--\eqref{923}, and the identity $\sfL\sfZ=\sfZ\sfLd$ in Lemma \ref{lem_Ld}, we successively extract the operator $\sfZ$ on the right-hand side of \eqref{924} to the outermost layer.
Then, we apply the operator $\sfZ^{-1}$.
In this way, we obtain
\begin{align*}
&~ \sfZ^{-1}\bQ\left( \sfZ\fh, \boxed{{}^{q+}\sfZ\fm},\boxed{{}^{q+}\sfZ\ft} \right)  \\
=&~ (\fh,q,\fm,1,\sfLd\ft) - (\fh,q,\fm,{}^{+}\sfLd\ft) -(\fh\boxplus(q,\fm),1,\sfLd\ft) + (\fh\boxplus(q,\fm),{}^{+}\sfLd\ft) \\
&~ -(\fh,1,\sfLd(\fm,q,\ft)) + (\fh,{}^{+}\sfLd(\fm,q,\ft))  +(\fh,1,\sfLd(\fm^{+q},\ft)) - (\fh,{}^{+}\sfLd(\fm^{+q},\ft)).
\end{align*}
The $\triangleleft$-minimal tuple appearing on the right-hand side of the above equation is $\left( \fh, {}^{q+}\fm^{+q}, \ft \right)$.
Indeed, it is the $\triangleleft$-minimal tuple appearing in the last term $(\fh,{}^{+}\sfLd(\fm^{+q},\ft))$.
Tuples appearing in other terms except for $(\fh\boxplus(q,\fm),{}^{+}\sfLd\ft)$ have greater depth.
If $\fh = \emptyset$, then $(\fh\boxplus(q,\fm),{}^{+}\sfLd\ft) = \bzero$; that is, no tuple appears.
We may assume $\fh \neq \emptyset$ below.
If $\ft = \emptyset$, then the tuple $(\fh\boxplus(q,\fm),{}^{+}\sfLd\ft)$ also has greater depth.
If $\ft \neq \emptyset$, then any tuple appearing in $(\fh\boxplus(q,\fm),{}^{+}\sfLd\ft)$ with the same depth as $\left( \fh, {}^{q+}\fm^{+q}, \ft \right)$ is $\triangleleft$-greater, because we compare entries from right to left.
\end{proof}

Now, we prove Theorem \ref{thm_Z2}.
\begin{proof}[Proof of Theorem \ref{thm_Z2}]
As in the proof of Theorem \ref{thm_Z1}, the induced linear operator $\overline{\LiA}$ is an isomorphism from $\Sw/\Qw$ to $\cZ_w^{(\bbF_q)}$.
Recall $\LiA(\sfZ\fs) = \LiAS(\fs)$ (see Lemma \ref{lem_Z}).
Therefore, it suffices to prove that 
\[
\left\{ \sfZ\fs + \Qw  \mid \fs \in \cB_w  \right\}
\]
is a basis of the quotient space $\Sw/\Qw$.

Using Lemma \ref{lem_last} and the minimal-tuple argument in the proof of Theorem \ref{thm_Z1}, we obtain
\begin{equation}\label{999}
\Sw/\sfZ^{-1}\Qw \subset \Span_{\bbF_q}\left\{ \fs + \sfZ^{-1}\Qw ~\Big|~ \fs \in \cI_w \setminus\{ \left( \fh, {}^{q+}\fm^{+q}, \ft \right) \mid \fh,\ft \in \cI, \fm \in \cI_{>0}  \}  \right\}.
\end{equation}
It is easy to see
\[
\cI_w \setminus\{ \left( \fh, {}^{q+}\fm^{+q}, \ft \right) \mid \fh,\ft \in \cI, \fm \in \cI_{>0}  \} = \cB_w
\]
and
\begin{align*}
\sum_{w=0}^{\infty} \#\cB_w x^w &= \frac{1-x}{1-2x+x^{q+1}} + \frac{1-x}{1-2x+x^{q+1}} \cdot (x^{q+1}+\cdots+x^{2q}) \cdot \frac{1-x}{1-2x+x^{q+1}}  \\
&= 1 + \frac{x(1-x^q)(1-2x+x^q)}{(1-2x+x^{q+1})^2}.
\end{align*}
So $\#\cB_w = d_w$ for $w \ge 1$.
Since $\dim_{\bbF_q}\Qw = 2^{w-1}-d_w$ by Propositions \ref{prop_half} and \ref{prop_other_half}, the inclusion in \eqref{999} must be an identity and 
\[
\Sw = \sfZ^{-1}\Qw \bigoplus \Span_{\bbF_q}\left\{ \fs \mid \fs \in \cB_w  \right\}.
\]
Applying the $\sfZ$ operator, we obtain
\[
\Sw = \Qw \bigoplus \Span_{\bbF_q}\left\{ \sfZ\fs \mid \fs \in \cB_w  \right\},
\]
which completes the proof.
\end{proof}

Theorems \ref{thm_Z1} and \ref{thm_Z2} together prove Theorem \ref{thm_Z}.

\vspace*{3mm}
\begin{flushright}
\begin{minipage}{148mm}\sc\footnotesize
J.\,Hu: School of Mathematical Sciences, Fudan University, Shanghai, China \\
{\it E-mail address}: \href{mailto:jinyuanhu2004@gmail.com}{{\tt jinyuanhu2004@gmail.com}}
\vspace*{3mm}
\end{minipage}
\end{flushright}

\begin{flushright}
\begin{minipage}{148mm}\sc\footnotesize
H.\,Huang: School of Mathematical Sciences, Peking University, Beijing, China  \\
{\it E-mail addresses}: \href{mailto:hanqing2007@gmail.com}{{\tt hanqing2007@gmail.com}}, \href{mailto:hanqing2007@stu.pku.edu.cn}{{\tt hanqing2007@stu.pku.edu.cn}} 
\vspace*{3mm}
\end{minipage}
\end{flushright}

\begin{flushright}
\begin{minipage}{148mm}\sc\footnotesize
L.\,Lai: School of Mathematical Sciences, Xiamen University, Fujian, China \\
{\it E-mail addresses}: \href{mailto:lilaimath@gmail.com}{{\tt lilaimath@gmail.com}}, \href{mailto:lilai@xmu.edu.cn}{{\tt lilai@xmu.edu.cn}} 
\vspace*{3mm}
\end{minipage}
\end{flushright}

\begin{flushright}
\begin{minipage}{148mm}\sc\footnotesize
K.\,Li: School of Mathematical Sciences, Peking University, Beijing, China\\
{\it E-mail addresses}: \href{mailto:likunyue2004math@163.com}{{\tt likunyue2004math@163.com}}, \href{mailto:2601110011@stu.pku.edu.cn}{{\tt 2601110011@stu.pku.edu.cn}}
\vspace*{3mm}
\end{minipage}
\end{flushright}

\end{document}